\documentclass[10pt,a4paper]{article}
\usepackage[latin1]{inputenc}
\usepackage{amsmath,amsthm, enumerate,hyperref}
\usepackage{amsfonts}
\usepackage{amssymb}
\newtheorem{lemma}{Lemma}

\newtheorem{prop}{Proposition}
\newtheorem{corollary}{Corollary}
\theoremstyle{remark}
\newtheorem{rmrk}{Remark}

\newtheorem{exmp}{Example}

\theoremstyle{definition}
\newtheorem{definition}{Definition}
\newtheorem{dfn}{Definition}

\newtheorem{assumption}{Assumption}
\usepackage[dvipsnames]{xcolor}
\newtheorem{thm}{Theorem}
\newtheorem*{thm*}{Theorem}

\usepackage{cleveref}
\crefname{equation}{equation}{equations}
\Crefname{thm}{Theorem}{Theorems}
\crefname{prop}{proposition}{propositions}
\crefname{exmp}{example}{examples}
\usepackage{tikz}
\usepackage[left=2cm,right=2cm,top=2cm,bottom=2cm]{geometry}
\usetikzlibrary{patterns}

\newcommand{\im}{\operatorname{im}}
\DeclareMathOperator{\ind}{ind}

 \newcommand{\cG}{\mathcal{G}}

 \newcommand{\cM}{\mathcal{M}}
 
 \newcommand{\cO}{\mathcal{O}}

 \newcommand{\cU}{\mathcal{U}} 
 \newcommand{\cV}{\mathcal{V}}

\newcommand{\C}{\mathbb{C}}

 \newcommand{\N}{\mathbb{N}}

 \newcommand{\R}{\mathbb{R}}

\newcommand{\X}{\mathbb{X}}

 \newcommand{\p}{\partial}

\makeatletter
\let\@fnsymbol\@arabic
\makeatother

\author{Andrei Agrachev, Stefano Baranzini, Ivan Beschastnyi \footnote{:The work of the third author was supported through the CIDMA Center for Research and Development in
Mathematics and Applications, and the Portuguese Foundation for Science and Technology (``FCT - Funda\c{c}\~ao para a Ci\^encia e a Tecnologia") within the project UIDP/04106/2020.}}
\title{Index theorems for graph-parametrized optimal control problems}

\begin{document}
\maketitle

\begin{abstract}
In this paper we prove Morse index theorems for a big class of constrained variational problems on graphs. Such theorems are useful in various physical and geometric applications. Our formulas compute the difference of Morse indices of two Hessians related to two different graphs or two different sets of boundary conditions. Several applications such as the iteration formulas or lower bounds for the index are proved.
\end{abstract}

\section{Introduction}
\label{sec:intro}
\subsection{Motivation}

\label{subsec:motivation}

The aim of this paper is to derive effective ways of computing the Morse index of second variation of constrained variational problems on graphs. Such problems can be conveniently formulated as optimal control problems. The results of this article can be used to study minimality and stability in a variety of geometrically and physically interesting problems.
	
Let us start with some examples which motivate the overall set-up in which we are working.
Given three points $a,b,c$ on a plane $\R^2$ place a point $d\in \R^2$ such that the sum of distances between $d$ and each of the points $a,b,c$ is minimal. It is well known that $d$ should be placed at the Fermat point, a fact central to the Steiner tree problem. In particular, each of the angles $\angle adb$, $\angle adc$, $\angle bdc$ should be of $120^\circ$.  
\begin{figure}[h]
		\begin{center}
		\input{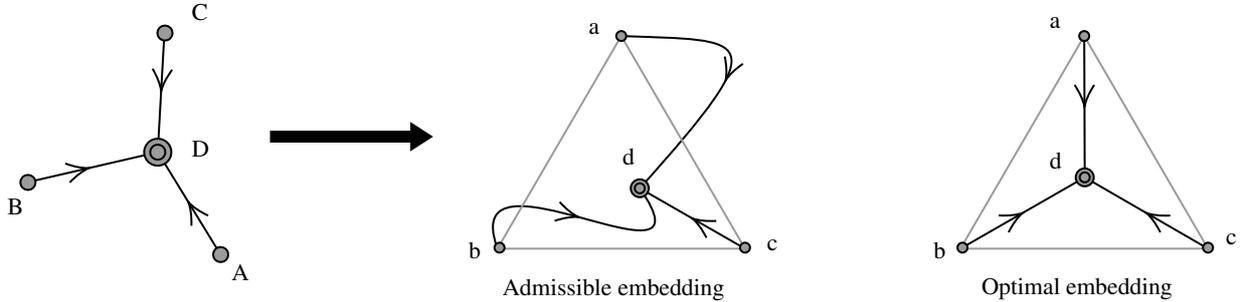}
		\caption{The graph $\mathcal{G}$ associated to the Fermat problem and two possible embeddings in $\mathbb{ R}^2$. \label{picture: Fermat graph} }
	\end{center}
	\end{figure}

We could have formulated the same problem in a slightly different, but equivalent way. Consider the tree graph $\cG$ in Figure \ref{picture: Fermat graph} and denote by $\cG_0 = \{A,B,C,D\}$ the set of its vertices and $\cG_1$ the set of its edges. We can parametrize each edge by the interval $[0,1]$, an operation which also assigns orientations to each of them. Then we would like to find a continuous map $F: \cG \to \R^2$ with smooth restrictions to each edge such that 
\begin{align*}
	F(A) &= a,\\
	F(B) &= b,\\
	F(C) &= c
\end{align*}
and
$$
\sum_{e\in\cG_1}l(F(e)) \to \min,
$$
where $l$ is the Euclidean length. So we have reformulated our problem as a minimal immersion of the graph $\cG$ into the Euclidean subspace $\R^2$. In this reformulation the following well-known notion plays a central role:
 
\begin{definition}
	A metric graph is a graph $\cG = (\cG_0,\cG_1)$, where $\cG_0$ is the set of vertices and $\cG_1$ is the set of parametrized edges. Each edge $e\in\cG_1$ is parametrized either by a finite interval $[0,l_e]$ for some $l_e >0$ or by $[0,+\infty)$.
\end{definition}

We can generalize the previous problem in several ways in order to encompass a great variety of situations commonly encountered in applications.
For example, we can consider more general metric graphs $\cG$, a manifold $M$ instead of $\R^2$, we can choose a different functional to minimize and, most importantly, we can assume that each edge satisfies a differential constraint. Let us consider some practical examples of this kind in order to see that an abstract approach is useful and allows to treat apparently different problems using the same framework.

Assume that we have a set of elastic rods soldered together to form a graph-like structure. Some vertices of this graph are assumed to be fixed, while others are free. What shapes can it take? This question arises often in civil engineering, for example, in construction of bridges, towers and various other structures (\Cref{picture:bridge}). They have some parts firmly fixed on the ground, which correspond to fixed vertices, while others are free to move in space. It is known that the elastic rods are extremal curves in certain constraint variational problems and stable configurations correspond to local minimizers of the bending energy~\cite{andrey_sachkov_elastica}. 

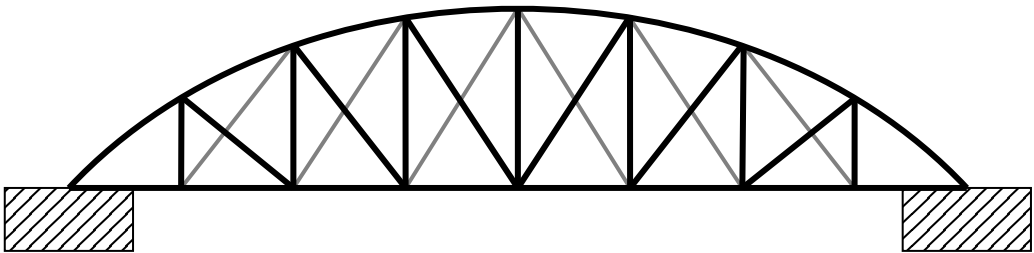
\begin{figure}[h]
		\begin{center}
 
\tikzset{
pattern size/.store in=\mcSize, 
pattern size = 5pt,
pattern thickness/.store in=\mcThickness, 
pattern thickness = 0.3pt,
pattern radius/.store in=\mcRadius, 
pattern radius = 1pt}
\makeatletter
\pgfutil@ifundefined{pgf@pattern@name@_0nbszopea}{
\pgfdeclarepatternformonly[\mcThickness,\mcSize]{_0nbszopea}
{\pgfqpoint{0pt}{0pt}}
{\pgfpoint{\mcSize+\mcThickness}{\mcSize+\mcThickness}}
{\pgfpoint{\mcSize}{\mcSize}}
{
\pgfsetcolor{\tikz@pattern@color}
\pgfsetlinewidth{\mcThickness}
\pgfpathmoveto{\pgfqpoint{0pt}{0pt}}
\pgfpathlineto{\pgfpoint{\mcSize+\mcThickness}{\mcSize+\mcThickness}}
\pgfusepath{stroke}
}}
\makeatother

 
\tikzset{
pattern size/.store in=\mcSize, 
pattern size = 5pt,
pattern thickness/.store in=\mcThickness, 
pattern thickness = 0.3pt,
pattern radius/.store in=\mcRadius, 
pattern radius = 1pt}
\makeatletter
\pgfutil@ifundefined{pgf@pattern@name@_gb2va1e0e}{
\pgfdeclarepatternformonly[\mcThickness,\mcSize]{_gb2va1e0e}
{\pgfqpoint{0pt}{0pt}}
{\pgfpoint{\mcSize+\mcThickness}{\mcSize+\mcThickness}}
{\pgfpoint{\mcSize}{\mcSize}}
{
\pgfsetcolor{\tikz@pattern@color}
\pgfsetlinewidth{\mcThickness}
\pgfpathmoveto{\pgfqpoint{0pt}{0pt}}
\pgfpathlineto{\pgfpoint{\mcSize+\mcThickness}{\mcSize+\mcThickness}}
\pgfusepath{stroke}
}}
\makeatother
\tikzset{every picture/.style={line width=0.75pt}} 

\begin{tikzpicture}[x=0.75pt,y=0.75pt,yscale=-0.8,xscale=0.8]

\draw  [pattern=_0nbszopea,pattern size=6pt,pattern thickness=0.75pt,pattern radius=0pt, pattern color={rgb, 255:red, 0; green, 0; blue, 0}] (40,200) -- (120,200) -- (120,239.58) -- (40,239.58) -- cycle ;
\draw  [pattern=_gb2va1e0e,pattern size=6pt,pattern thickness=0.75pt,pattern radius=0pt, pattern color={rgb, 255:red, 0; green, 0; blue, 0}] (600,200) -- (680,200) -- (680,239.58) -- (600,239.58) -- cycle ;
\draw [line width=2.25]    (80,200) -- (640,200) ;
\draw [line width=2.25]    (80,200) .. controls (219,49.97) and (500,49.47) .. (640,200) ;
\draw [line width=2.25]    (150.22,142.76) -- (220,200) ;
\draw [line width=2.25]    (220,110) -- (290,200) ;
\draw [line width=2.25]    (289.89,92.76) -- (360,200) ;
\draw [line width=2.25]    (429.53,92.76) -- (360,200) ;
\draw [line width=2.25]    (500,110) -- (430,200) ;
\draw [line width=2.25]    (570.04,143.76) -- (500,200) ;
\draw [color={rgb, 255:red, 0; green, 0; blue, 0 }  ,draw opacity=0.5 ][line width=1.5]    (500,110) -- (528.87,147.12) -- (570,200) ;
\draw [color={rgb, 255:red, 0; green, 0; blue, 0 }  ,draw opacity=0.5 ][line width=1.5]    (429.53,92.76) -- (500,200) ;
\draw [color={rgb, 255:red, 0; green, 0; blue, 0 }  ,draw opacity=0.5 ][line width=1.5]    (360.04,87.09) -- (430,200) ;
\draw [color={rgb, 255:red, 0; green, 0; blue, 0 }  ,draw opacity=0.5 ][line width=1.5]    (360.04,87.09) -- (290,200) ;
\draw [color={rgb, 255:red, 0; green, 0; blue, 0 }  ,draw opacity=0.5 ][line width=1.5]    (289.89,92.76) -- (220,200) ;
\draw [color={rgb, 255:red, 0; green, 0; blue, 0 }  ,draw opacity=0.5 ][line width=1.5]    (220,110) -- (150,200) ;
\draw [line width=2.25]    (150.22,142.76) -- (150,200) ;
\draw [line width=2.25]    (220,110) -- (220,200) ;
\draw [line width=2.25]    (289.89,92.76) -- (290,200) ;
\draw [line width=2.25]    (360.04,87.09) -- (360,200) ;
\draw [line width=2.25]    (429.93,91.76) -- (430,200) ;
\draw [line width=2.25]    (500.82,110.83) -- (500,200) ;
\draw [line width=2.25]    (570.04,143.76) -- (570,200) ;

\end{tikzpicture}
		\caption{An example of a truss bridge, which can be viewed as a number of connected elastic beams forming a graph. \label{picture:bridge} }
	\end{center}
	\end{figure}

Another interesting example comes from quantum mechanics. A quantum graph is a metric graph $\cG$ with a possibly non-linear Schr\"odinger equation defined on it. A ground state $\psi: \cG \to \C$ of a quantum graph is a global minimizer of the functional
$$
\sum_{e\in \cG_1}\int_0^{l_e} \frac{|\dot{\psi}_e|^2}{2} - \frac{|\psi_e|^\alpha}{\alpha}dt \to \min ,
$$
under the constraint of fixed total mass $\mu$: 
$$
\sum_{e\in \cG_1}\int_0^{l_e} |\psi_e|^2 dt = \mu ,
$$
{together with} a regularity condition $\psi\in H^1(\cG)$:
$$
\sum_{e\in \cG_1}\int_0^{l_e} |\psi_e|^2 + |\dot{\psi}_e|^2 dt < \infty,
$$
and certain boundary condition on the value of $\psi$ at the vertices $\cG_0$ which are usually taken to be Dirichlet, Neumann or Kirkchoff (or a combination of the above). Here $\psi_e$ are restrictions of $\psi$ to the edge $e$, $l_e$ is the length of edge and $2\leq \alpha \leq 6$ is a constant.

There is an extensive literature concerning quantum graphs. In particular regarding existence of global minimizers for the NLS energy, see for example \cite{NLS1}, \cite{NLs3}, \cite{NLS2}, \cite{NLS4}  (or \cite{overviewNLS} and \cite{noteNLS} for overviews of the results) and references therein. Also in situation that do not fit in our current framework, for example \cite{NLSperiodic} where the same problem is discussed on a graph with an infinite number of edges. For the linear problem when $\alpha =2$ the book~\cite{berkolaiko} is a good source.

Another important application comes from quantum physics. In the perturbative approach to quantum mechanics and quantum field theory via the path integral method, a formal analogue of the stationary phase method is used. This formula requires to know the index and a suitable generalization of the determinant of the second variation at a critical point~\cite{dunne,ludewig}. The index is computed in the current paper, while the determinant will be investigated in a forthcoming paper.

\subsection{Problem statement}
\label{subsec:statement}

In order to transmit better the ideas and simplify the proofs we will make several technical assumptions, starting with the following one.

\begin{assumption}
	\label{ass:graph}
	Graph $\cG$ has a finite number of edges.
\end{assumption} 

The class of minimization problems we are interested is described in the next definition.

\begin{definition}
	Given a metric graph $\cG = (\cG_0,\cG_1)$ and a manifold $M$ consider the following data:
	\begin{enumerate}
		\item \textit{Control constraints} $U_e$ parametrized by the elements of the edge set $e\in \cG_1$, which are subsets $U_e \subset \R^{k_e}$ for some $k_e \in \N$;
		\item \textit{Families of time-dependent complete vector fields} $f_{t,u}^e\in Vec(M)$ parametrized by the elements of the edge set $e\in \cG_1$ and controls $u\in U_e$;
		\item \textit{Lagrangians} $\ell^e: [0,l_e]\times M \times U_e \to \R$ parametrized by the edge set $e\in \cG_1$.
		\item \textit{Boundary conditions} given by a subset $N\subset M^{| \cG_0|}$.
	\end{enumerate}
	\textit{A graph-parametrized optimal control problem} is the problem of finding a continuous map $q: \cG \to M$ with almost everywhere differentiable restrictions to edges $q_e: [0,l_e]\to M$, such that it minimizes the following functional
	\begin{equation}
		\label{eq:graph_functional}
		\varphi[u]=\sum_{e\in \cG_1} \int_0^{l_e} \ell^e(q_e(t),u_e(t))dt \to \min
	\end{equation}
	subject to constraints
	\begin{equation}
		\label{eq:graph_control}
		\dot{q}_e = f^e_{t,u_e(t)}(q_e), \qquad u_e\in L^\infty([0,l_e],U_e),
	\end{equation}
	\begin{equation}
		\label{eq:graph_boundary}
		\qquad (q(v_1),q(v_2),\dots, q(v_{|\cG_0|}) \in N. 
	\end{equation}
	where $v_1,\dots,v_{|\cG_0|}$ are vertices in $\cG_0$ indexed by integers.
\end{definition} 
All of the examples from \Cref{subsec:motivation} can be formulated as graph-parametrized optimal control problems. The goal of this article is to study the second variation of such minimization problems and characterize local minimizers. Local minimizers play an important role in modelling of various physical phenomena, since they usually correspond to stable configurations observable in nature, which makes them relevant even if there are no global minima. 

The next assumption allows to avoid various regularity technicalities.
\begin{assumption}
	\label{ass:regularity}
	\begin{enumerate}
		\item $N$ is an embedded submanifold;
		\item $U_e = \R^{k}$ for some $k\in \N$ and all $e\in \cG_1$;
		\item Vector fields $f_u^e$ and functions $\ell^e$ are jointly smooth in the space variables $q$ and in the control variables $u$, piecewise smooth in $t$ for all $e\in \cG_1$.
	\end{enumerate}
\end{assumption}

We can reformulate a graph-parametrized optimal control problem~\eqref{eq:graph_functional}-\eqref{eq:graph_boundary} as an equivalent standard optimal control problem with non-fixed boundary conditions of the form:
\begin{equation}
	\label{eq:control}
	\dot{q} = f^t_{u(t)}(q), \qquad q\in M, \qquad u \in L^\infty([0,1],\R^k);
\end{equation}
\begin{equation}
	\label{eq:boundary}
	(q(0),q(1))\in N \subset M\times M;
\end{equation}
\begin{equation}
	\label{eq:functional}
	\varphi[u]=\int_0^1 \ell(t,q(t),u(t))dt \to \min.
\end{equation}
We will give the precise algorithm to do so later, in Subsection~\ref{subsec:reduction_to_1d}. Note that~\eqref{eq:control}-\eqref{eq:functional} by itself is a special case of a graph-parametrized optimal control problem where the metric graph $\cG$ is the interval $[0,1]$. 

The usual approach for identifying local minimizers of this problem can be roughly described as follows. First one applies first order minima conditions and identifies critical points of the functional $\varphi$. These critical points are called \textit{extremal curves}. After that one studies second order conditions.
\begin{definition}
\label{def:hessian}
Suppose that $u$ is a critical point of the functional $\varphi$ restricted to the space of admissible variations (i.e. variation satisfying conditions \eqref{eq:control}-\eqref{eq:boundary}). Assume that the space of admissible variation is a  smooth manifold in a neighbourhood of $u$, the Hessian (or the second variation) is a quadratic form $Q$ defined on the tangent space to the space of admissible variations. It is given by the second derivative of the functional $\varphi$.
\end{definition}
Often, in constrained variational problems, the space of admissible variations is not a manifold. For this reason we will make additional assumptions in order to guarantee that this space is at least locally smooth in a neighbourhood of a critical point.

Quadratic form $Q$ encodes information about minimality of a given extremal curve in its kernel and negative inertia index $\ind^- Q$. In the classical calculus of variations a necessary and sufficient condition for an extremal curve to be a local minimizer is $\ind^- Q = 0$, whenever the second variation is non-degenerate. Note that this is not always true for constrained variational problems. Sometimes a critical point stops being a local minimizer only when the inertia index of $Q$ exceeds a certain threshold (see Theorem 20.3 in \cite{bookcontrol}). Thus we need to find good algorithms for computing $\ind^- Q$. This paper provides several efficient ways of doing this in the context of graph-parametrized problems.

This reduction allows us to focus completely on \eqref{eq:control}-\eqref{eq:functional}. That problem has a relatively simple characterisation of critical points, known as the Ponrtyagin maximum principle (PMP). In order to describe the PMP in this setting we introduce a family of Hamiltonian functions on $T^* M$
$$
h^t_u(\lambda)=\langle \lambda, f^t_u(q)\rangle - \nu \ell^t(q,u), \qquad q = \pi(\lambda),
$$
where $\nu \in \{0,1\}$ and for any submanifold $N \subseteq M \times M$ define 
\begin{equation}
\label{eq:ann_def}
A(N) = \{(\lambda_0,\lambda_1)\in T^*M \times T^*M\, : \, \lambda_0(X_0) = \lambda_1(X_1), \,\forall (X_0,X_1) \in T_{\pi( \lambda_0,\lambda_1)}N\},
\end{equation}
which should be thought of as the annihilator of $N$.

\begin{thm}
	A curve $q:[0,1]\to M$ is an extremal curve if there exists a control function $u\in L^\infty[0,1]$ and a curve $\lambda:[0,1]\to T^*M$ such that for almost all $t\in [0,1]$
	\begin{enumerate}
		\item the curve $q$ is the projection of $\lambda$: 
		$$
		q(t) = \pi (\lambda(t));
		$$
		\item $\lambda$ satisfies the following Hamiltonian system:
		$$
		\frac{d \lambda}{dt} = \vec{h}_{u}^t(\lambda);
		$$
		\item the control $u$ is determined by the maximum condition:
		$$
		h^t_{u(t)}(\lambda(t)) = \max_{u\in \R^k} h^t_u( \lambda(t));
		$$
		\item the non-triviality condition holds: $(\lambda(t),\nu)\neq (0,0)$;
		\item and the transversality conditions holds:
		\begin{equation}
			\label{eq:transversal}
			(\lambda(0),\lambda(1)) \in A(N).
		\end{equation}
	\end{enumerate}
\end{thm}

It is known that extremal curves in optimal control problems can have very different behaviour. They are usually separated into several classes, such as regular, singular, bang-bang and others~\cite{schattler-ledzewicz-book,osmolovskii,aronna_bang_singular,agrachev_stefani_bang}. We will now list the last pair of assumptions that will allow us to focus on a rather broad class of extremals for which a good geometric description of the second variation is possible.

\begin{assumption}
	\label{ass:hamilton}
	The maximized Hamiltonian
	$$
	H^t(\lambda) = \max_{u\in U} h_u^t(\lambda), \qquad \lambda \in T^*M
	$$
	is well-defined and $C^2$ on $T^*M \times[0,1]$. 
\end{assumption}

\begin{assumption}
	\label{ass:legendre}
	If $\lambda: [0,1]\to T^*M$ is an extremal satisfying PMP with control $u\in L^\infty([0,1],\R^k)$, then it satisfies the \textit{strong Legendre condition}. Which means that there exists a constant $c>0$ such that 
	\begin{equation}
		\label{eq:strong_legendre}
		\left.\frac{\p^2 h^t_u}{\p u^2}(v,v) \right|_{u=u(t)} \leq - c \|v\|^2.
	\end{equation}  
Assumption~\ref{ass:hamilton} will permit us to state the results in a  simple form using the Hamiltonian flow of $\vec{H}^t$, while Assumption~\ref{ass:legendre} guarantees that the quadratic form $Q$ from \Cref{def:hessian} has finite negative index, and that small arcs of a given extremal curve are local minimizers~\cite[Theorem 20.1]{bookcontrol}.
	
\end{assumption}

\subsection{Main results and structure of the paper}
\label{subsec:results}
We are now ready to formulate and discuss the main results of this paper. Consider an extremal $\lambda$ of an optimal control problem~\eqref{eq:control}-\eqref{eq:functional}, which satisfies a Hamiltonian system
\begin{equation}
	\label{eq:hamilton}
	\dot{\lambda} = \vec{H}^t(\lambda)
\end{equation}
under the Assumption~\ref{ass:hamilton}. This Hamiltonian vector field $\vec{H}^t$ generates a flow $\Psi_t: T^*M \to T^*M$.  Denote by
$$
\Gamma(\Psi_t) = \{(\lambda,\Psi_t(\lambda)): \lambda \in T^*M\} \subset T^*M \times T^*M
$$
its graph, which is a smooth submanifold of the product space $T^*M \times T^*M$. When $t=1$ we simply write $\Gamma(\Psi) = \Gamma(\Psi_1)$. 
We have the following main index theorem for the optimal control problem on the interval.
\begin{thm}
	\label{thm: comparison theorem}
	Let $\lambda:[0,1]\to T^*M$ be an extremal for~\eqref{eq:control}, \eqref{eq:functional} and simultaneously for two different boundary conditions $N$ and $\tilde{N}$ in~\eqref{eq:boundary}. Let $Q_N$ and $Q_{\tilde{N}}$ be the two quadratic forms for the second variation corresponding to the two boundary conditions. Denote $\underline{\lambda} = (\lambda(0),\lambda(1))$.  Then under the Assumptions~\ref{ass:graph}-\ref{ass:legendre} the negative inertia indices $\ind^- Q_N$, $\ind^- Q_{\tilde{N}}$ are finite and
	\begin{align}
		\label{eq:main_index_formula}
		\ind^-Q_{\tilde N} - \ind^{-} Q_{N} &= i\big(T_{\underline \lambda}A(N),T_{\underline{\lambda}}\Gamma(\Psi),T_{\underline \lambda}A(\tilde N)\big)+\dim(T_{\underline{\lambda}}N \cap T_{\underline{\lambda}}\tilde{N}) - \dim T_{\underline{\lambda}}N + k_0
	\end{align}
where $k_0 = \dim(T_{\underline \lambda} A(N) \cap T_{\underline{\lambda}}\Gamma(\Psi)) -   \dim(T_{\underline \lambda} A(N)\cap  T_{\underline{\lambda}}\Gamma(\Psi) \cap T_{\underline \lambda} A(\tilde N))$.
\end{thm}

Only one term on the right hand-side still requires an explanation. Here $i(T_{\underline \lambda}A(N),T_{\underline{\lambda}}\Gamma(\Psi),T_{\underline \lambda}A(\tilde N))$ is the negative Maslov index of the triple of Lagrangian subspaces indicated in the bracket of the symplectic space $T_{\lambda(0)}(T^*M)\times T_{\lambda(1)}(T^*M)$ with the form $\big(-\sigma_{\lambda(0)}\big)\oplus \sigma_{\lambda(1)}$. A precise definition of Lagrangian spaces and of the Maslov index used in this paper will be given in \Cref{section: maslov index}. For now it is enough to know that it is a certain symplectic invariant of the triple of subspaces, which can be computed in an explicit algebraic way.

A relevant example to keep in mind when the conditions of the theorem above are met is when $N \subset \tilde{N}$. In this case if $\lambda$ satisfies the transversality conditions for $\tilde{N}$, then it satisfies the transversality conditions for $N$ automatically. In particular, $N$ can be just the product of two points $N=\{q_0\}\times \{q_1\}$, for which transversality conditions are trivially satisfied. This allows us to compute the difference between indices of the second variation for the problem with moving and fixed end-points.

Theorem~\ref{thm: comparison theorem} has many interesting applications and allows to have a fresh view even on some classical results. When we consider a graph-parametrized problem~\eqref{eq:graph_functional}-\eqref{eq:graph_boundary} and reformulate it as problem on an interval~\eqref{eq:control}-\eqref{eq:functional}, the structure of the graph $\cG$ is completely encoded in the boundary condition $N$. It is often the case that a single critical point satisfies two graph-parametrized problems with different boundary constraints or even with different underlying graphs.

For example, we can introduce an extra vertex on an edge of a graph, and assume that this vertex is free. This obviously does not change the possible critical points. However, now, we can compare it to a problem where the new vertex is fixed. The simplest example of this technique would be to consider a fixed boundary optimal control problem on an interval~\eqref{eq:graph_functional}-\eqref{eq:graph_boundary} with $N=\{q_0\}\times \{q_1\}$, introduce several free vertices in the interior of the interval and compare with a problem where each of the vertices is fixed as depicted in Figure~\ref{fig:discrete}.

\begin{figure}[h]
\begin{center}
	\input{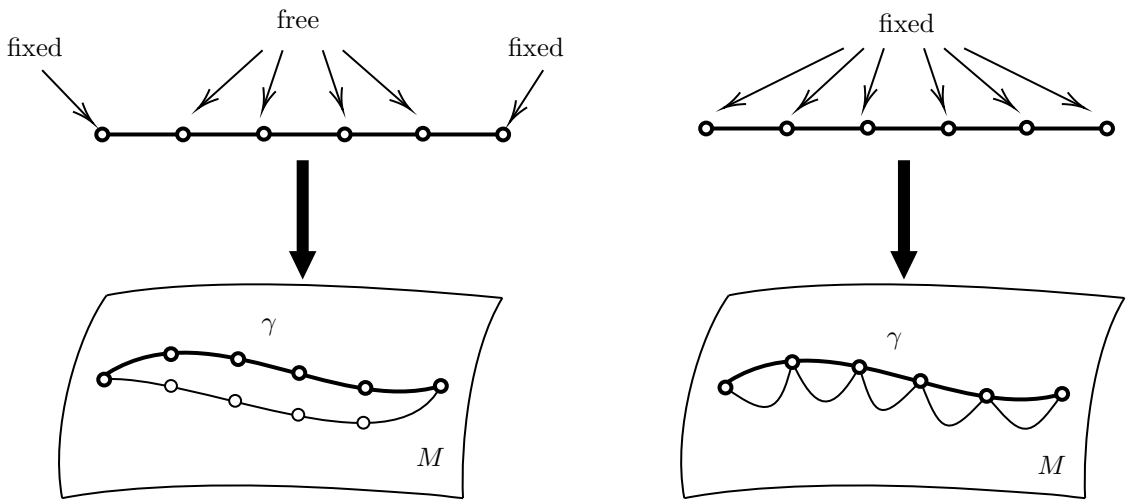}	
\end{center}
	\caption{Variation of $\gamma$ in the original problem and a problem with extra fixed vertices.\label{fig:discrete}}
\end{figure}

In order to formulate the next result, we need the definition of conjugate times and conjugate points.

\begin{definition}
	Given $\mu \in T^*M$, denote by $\Pi_{\mu}$ the vertical subspace $T_{\mu}(T_{\pi(\mu)}^*M)$. Given an extremal $\lambda:[0,1]\to T^* M$ of an optimal control problem~\eqref{eq:control}-\eqref{eq:functional}, we say that a moment of time $t\in[0,1]$ is \textit{conjugate} if the map
	$$
	\pi_* \circ (\Psi_t)_*|_{\Pi_{\lambda(0)}}
	$$
	has a kernel. The corresponding point $q(t) = \pi(\lambda(t))$ is said to be a \textit{conjugate point}. 
\end{definition}

To simplify notation we will denote by $\Theta_t = (\Psi_t)_*$ the differential (or the \emph{linearisation}) of the extremal flow. It is straightforward to check that the tangent space  $T_{\underline{\lambda}}\Gamma(\Psi)$ is actually the graph of the linear map $\Theta_t: T_{\lambda} T^*M \to T_{\Psi_t(\lambda)}T^*M$. We will denote by $\Gamma(\Theta_t)$ said graph and by $\Gamma(\Theta)$ the graph at time $t=1$.

A consequence of Theorem~\ref{thm: comparison theorem} is the following result. 

\begin{thm}[Discretization]
	\label{thm:discrete}
	Let $\lambda:[0,1]\to T^*M$ be an extremal for~\eqref{eq:control}-\eqref{eq:functional} with $N=\{q_0\}\times \{q_1\}$ and let $\Xi =\{t_0,\dots, t_n\}$ be a partition of $[0,1]$. Denote by $\Theta_{i+1,i}$ the restriction to the interval $[t_i,t_{i+1}]$ of the differential of the extremal flow in~\eqref{eq:hamilton}. The following formula holds:
	\begin{equation}
		\label{eq: discretization formula}
		\ind^- Q \ge \sum_{i=0}^{n-1} i \big(		
		\Theta_{i+1,i}^{-1}(\Pi_{i+1}),\Pi_{i},\Theta_{i,i-1} \circ \dots \circ  \Theta_{1,0}(\Pi_{0})\big), 
	\end{equation} 
	where $\Pi_i = T_{\lambda(t_i)}(T^*_{\pi(\lambda(t_i))}M) \simeq T^*_{\pi(\lambda(t_i))}M$. Moreover, equality holds if $\max_i|t_{i+1}-t_i|$ is sufficiently small and no $t_i$ is a conjugate time.
\end{thm}

As previously discussed a necessary condition for minimality under Assumptions~\ref{ass:graph}-\ref{ass:legendre} is $\ind^- Q = 0$. For this reason a necessary condition for minimality of a critical point would be the equality to zero of the right hand side in formula~\eqref{eq: discretization formula}. In practice, this allows to determine non-optimal solutions and greatly reduce the number of candidates for the minimal solution.

Another example of this type is given by the $k$-th iterate $\gamma^k$ of a periodic extremal trajectory $\gamma$. If $\gamma$ has period $T$, we can view $\gamma^k$ as a periodic trajectory of period $kT$. Hence it is a graph-parametrized problem with the graph having one edge of length $kT$ and a single vertex. We can add $k$ more equispaced vertices and compare the problem to the $k$ copies of smaller circle graphs which correspond to $\gamma$ as depicted in Figure~\ref{fig:iteration}. 
\begin{figure}[h]
\begin{center}
\input{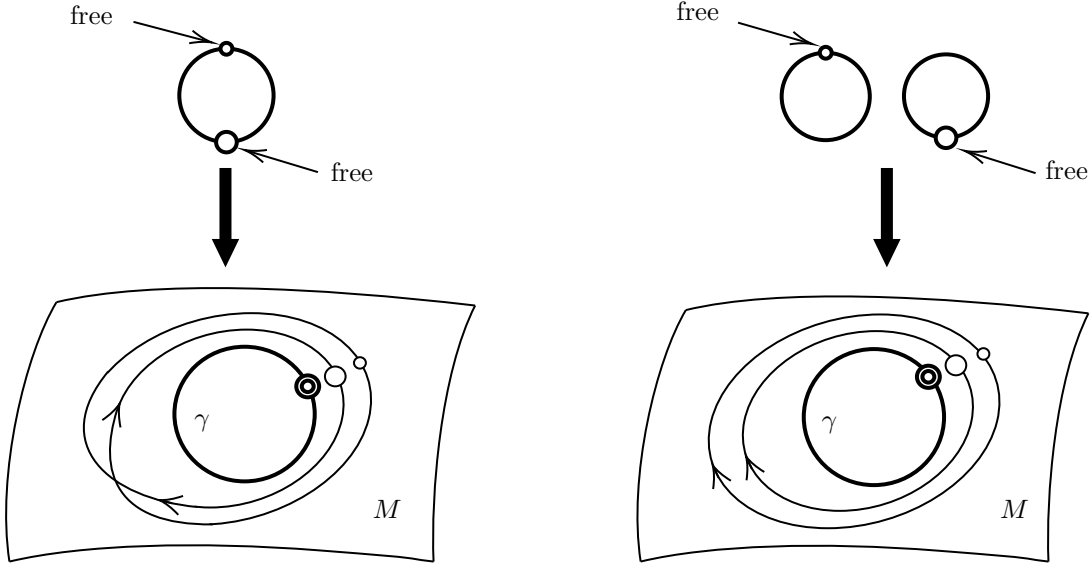}
\end{center}
\caption{Variations of a periodic extremal trajectory $\gamma$ as a periodic  trajectory ran twice (left) and as two separate periodic trajectory (right).\label{fig:iteration}}
\end{figure}
We denote by $\Theta$ the linearization of the flow \cref{eq:hamilton} along $\gamma$ at the moment of time $t=T$. Again an application of Theorem~\ref{thm: comparison theorem} gives us the following iteration formulae:
\begin{thm}[Iteration Formulae]
	\label{thm:iteration}
	Let $k > 1$ and $\gamma$ be an extremal curve in the problem~\eqref{eq:control}-\eqref{eq:functional} with $N = \Delta$, where $\Delta\subset M\times M$ is the diagonal of $M$. The index of the $k-$th iteration of $\gamma$ as a periodic trajectory satisfies:
	\begin{equation*}
		\begin{split}
			\ind^-Q_{\gamma^k} - k\, \ind^-Q_{\gamma}
			&= \sum_{j=1}^k i \big(\Gamma(\Theta^{j-1}),T_{\underline{\lambda}}A(\Delta),\Gamma(\Theta^{j})\big)-\dim(M) + \dim(\ker(\Theta^{j-1}-1))\\
			 &= \sum_{j=1}^{k-1} \dim(M)-\dim(\ker (\Theta-\omega^j))-i \big(\Gamma(\Theta),T_{\underline{\lambda}}A(\Delta),\Gamma(\omega^j\Theta)\big) .
		\end{split}
	\end{equation*} 
	Where $\omega$ is a primitive $k$-th root of unity.
\end{thm}
The first equality is much in the spirit of \cite{CushmannDuistermaatIteration} and \cite{DuistermaatIntersectionIndex}. 
The second equality in the formula is obtained using a complexified version of Maslov index described in \Cref{section: maslov index}, yielding a result very similar to Bott's original approach given in \cite{Bott}. 

The theorems above are examples of index formulas, which try to encode the information about the index of the second variation of variational problems in terms of geometric quantities such as the Maslov index above. In the context of variational problems on 1D objects such as curves or graphs, it is possible to reduce the problem of studying index of a linear operator on an infinite-dimensional space to the study of non-autonomous linear ODEs in finite dimensional spaces. There exist various analogues of this result. For example, in the context of classical calculus of variations an infinite-dimensional version of the Morse index formulas was proven by several groups in works~\cite{DuistermaatIntersectionIndex,swanson,cox_jones_schroedinger,cox_jones_morse_bounded_domains}. In the case of strongly indefinite problems the index formulas are replaced by spectral flow theorems, which are valid both in finite~\cite{portaluri_morse_geodesics,piccone_conjugate,waterstraat_kmorse} and infinite dimensions~\cite{portaluri_waterstraat}. There are also various approaches to infinite-dimensional Morse homology~\cite{abbondandolo_infinite_morse}. A very general index theorem for optimal control problems was proven by the first and third author in~\cite{beschastnyi_morse}, which encompasses many separate classical constructions for various types of extremals~\cite{schattler-ledzewicz-book,osmolovskii,aronna_bang_singular,agrachev_stefani_bang,beschastnyi_1d}. 

A variant of Theorem~\ref{thm: comparison theorem} was proven by Baryshnikov in~\cite{baryshnikov}. His formula is true in the generic case for a graph with a finite number of edges. In the generic picture, the various intersection terms in~\eqref{eq:main_index_formula} disappear. The authors of~\cite{latushkin_index_on_graphs} study the Morse index of Schr\"odinger operators on graphs. By reducing the problem to an interval, they provide a Morse index formula as the Maslov index of a curve in a Lagrangian Grassmanian of a sufficiently big dimension. Our formulas are different from those in several significant ways. They apply to constrained variational problems, they do not require any genericity assumptions, they work with more general boundary conditions and separate the contribution to the index coming from varying the edges and the contribution to index coming from varying vertices.

The paper has the following structure. It is divided into two parts. The first part consists of~\Cref{section: preliminaries} and \Cref{section: proof applications}. \Cref{section: preliminaries} contains all the necessary information regarding symplectic geometry, Lagrangian spaces and Maslov index. This is enough to use Theorem~\ref{thm: comparison theorem} as black box.
In \Cref{section: proof applications} we focus on applications for the graph-parametrized problems. We prove Theorems~\ref{thm:discrete} and~\ref{thm:iteration} as well as a formula by Baryshnikov which reduces the dimensionality in \Cref{thm: comparison theorem} using a filtration of the vertices. This section relies only on Theorem~\ref{thm: comparison theorem} and properties of the Maslov index. The reader who only wants to understand how to compute the index and apply the main theorems can focus just on the first part of the paper. The second part is \Cref{section: proof comparison}, where we prove the main Theorem~\ref{thm: comparison theorem}. Our strategy will be to reduce the problem~\eqref{eq:control}-\eqref{eq:functional} to a problem with fixed end-points. We then derive the Jacobi equation and use it to give a geometric interpretation of various terms entering in the formula for the second variation.

	\section{Preliminaries}
	\label{section: preliminaries}
\subsection{Symplectic geometry}

\label{subsec:symplectic}

A \emph{symplectic} vector space is a finite dimensional vector space $\Sigma$ with a non degenerate skew-symmetric bilinear form $\sigma$ (the \emph{symplectic form}). A \emph{symplectic} manifold $M$, is a manifold whose tangent space $TM$ is endowed with a symplectic structure at each point (i.e $M$ together with a closed non degenerate $2-$form).

Cotangent spaces of smooth manifolds are always endowed with a symplectic structure which is given by the so called \emph{tautological} form. Call $\pi: T^*M\to M$ the canonical projection. Take $\lambda \in T^*M$ and define the $1-$form $s_\lambda (X) = \lambda(\pi_*X)$. One can check that $d s$ is non degenerate and thus $(T^*M,d s)$ is a symplectic manifold. The $2-$form $\sigma = ds$  on $T^*M$ is called the  \textit{canonical symplectic form}.

A linear map $\Psi$ between symplectic vector spaces $(\Sigma_1,\sigma_1) \to (\Sigma_2,\sigma_2)$ is called a (\emph{linear}) symplectomorphism if $\Psi^* \sigma_2 = \sigma_1$. A diffeomorphism is a symplectomorphism if its differential is a linear symplectomorphism.

A natural way to obtain diffeomorphisms is trough flows. Given a (complete) vector field $X$ one obtains a family of diffeomorphisms $\Phi_t$ by solving the ODE system 
\begin{align*}
\dot{\Phi}_t &= X(\Phi_t), \\
\Phi_0 &= Id. 
\end{align*} 
In a similar way one can produce symplectomorphisms using special classes of vector fields:
\emph{Hamiltonian} and \emph{symplectic} fields. A vector field $X$, is \emph{Hamiltonian} if there is a smooth function $H$ such that $dH(Y) = \sigma(Y,X)$ for all smooth vector fields  $Y$. $H$ is called \emph{Hamiltonian} function and $X $ is often denoted by $\vec{H}$. A vector field for which we can find a Hamiltonian only locally, i.e. in a neighbourhood  of every point, is called \emph{symplectic}. The flow of Hamiltonian and symplectic vector fields is always a one parameter group of symplectomorphisms.

Given a subspace $W$  of a symplectic vector space $\Sigma$, we can define its skew-orthogonal complement $W^\perp$ using the symplectic form. 
$$
W^{\perp} = \{u \in \Sigma: \sigma(u,w)=0, \,  \forall w \in W\}.
$$
Since the symplectic form is non degenerate, $\dim(W)+\dim(W^{\perp}) = \dim(\Sigma)$ and $(V+W)^{\perp} = V^{\perp} \cap W^{\perp}$.

Inside a symplectic vector space we distinguish the following classes of subspaces:
\begin{itemize}
	\item \textbf{Isotropic} $V$ such that $\sigma (v,w) =0$, $\forall u,v \in V$, i.e. $V\subseteq V^{\perp}$;
	\item \textbf{Lagrangian} $V$ isotropic and maximal, i.e. $V = V^\perp$;
	\item \textbf{Coisotropic} $V$ such that $V^{\perp}\subseteq V$.
\end{itemize} 

Lagrange subspaces are extremely important in symplectic geometry, their collection is a compact manifold called Lagrange Grassmannian. It is denoted by 
$$
Lag(\Sigma)=\{V \subseteq \Sigma : V = V^\perp \}.
$$ 
If $\dim(\Sigma) = 2n$ its dimension is $n(n+1)/2$. 

The following examples of Lagrangian subspaces are often considered:
\begin{exmp}
	If $\Psi$ is a linear symplectomorphism, the graph of $\Psi$ is a subspace of the product space $\Sigma_0\bigoplus \Sigma_1$. The product space can be endowed with a symplectic structure considering $(-\sigma_0)\oplus \sigma_1$. Graphs of symplectomorphisms are always Lagrangian subspaces with this choice.
\end{exmp}
\begin{exmp}
	If $N\subseteq M$ is a submanifold of a smooth manifold we can always consider the following submanifold of the cotangent bundle $T^*M$:
	\begin{equation}
		\label{def: eq annullatore}
		A(N):= \{\lambda \in T^*M : \pi(\lambda)\in N, \lambda(X) =0, \,\forall\, X \in T_{\pi{(\lambda)}}N  \}.
	\end{equation} 
	
	The tangent space to the annihilator at a point $\lambda$ is a Lagrangian subspace of $T_\lambda(T^*M)$, which means that $A(N)$ is a Lagrangian submanifold. 
	
	If $M$ is a product space i.e. $M = M_0 \times M_1$, the annihilator of a submanifold $N$ is a Lagrangian submanifold only with respect to $\sigma_0\oplus\sigma_1$, which coincides with the canonical symplectic form on $T^*(M_0\times M_1)$.
	
	To get a Lagrangian submanifold of $(T^*M_0 \oplus T^*M_1,(-\sigma_0)\oplus \sigma_1)$ one has to change the sign to the first or the second covector and thus consider the submanifold $A(N)$ defined in~\eqref{eq:ann_def}. Notice that this distinction is unnecessary when $N$ itself is  a product, i.e. when $N = N_0 \times N_1$.
\end{exmp} 
\begin{exmp}
	A particular instance of the example above is the vertical fibre i.e. the tangent space to $T^*_q M$. This space can be characterized as the annihilator of the point $q$ or as the kernel of the natural projection, $\ker \pi_*$:
	\begin{equation}
		\label{eq: fibre definition}
		\Pi_\lambda =\{\xi \in T_\lambda (T^*M) : \pi_*(\xi) =0\}.
	\end{equation}
\end{exmp}
\begin{exmp}
\label{exmp:standard}
	Consider as symplectic space $\mathbb{R}^{2n} = \{x = (p,q): p,q \in \mathbb{R}^n\}$ with the \textit{canonical} symplectic form $\sigma(x,x') = \langle p,q'\rangle -\langle q,p'\rangle$.
	Using the Euclidean scalar product we can represent $\sigma$ as:
	\begin{equation}
	\label{eq: matrix J}
	\sigma(x,x') = \langle J x, x'\rangle \quad \text{ where } \quad 
	J = \begin{pmatrix}
		0 &-1 \\ 1&0
	\end{pmatrix}
	\end{equation}
	The two subspaces $B = \{p=0\}$ and $\Pi = \{q =0\}$ are Lagrangian and any subspace of the form $V_S =\{(q,Sq)\}$ and $V'_S = \{(Sp,p)\}$ is Lagrangian provided that $S = S^*$.
\end{exmp}

	It turns out that if we take two transversal Lagrangian subspaces $L_0$ and $L_2$ there always exists a choice of basis, for which $L_0$ is $B$ and $L_2$ is $\Pi$, and such that the symplectic form $\sigma$ has the canonical form as in \cref{eq: matrix J} (see Theorem 1.15 in~\cite{bookgosson}). These coordinates are sometimes called \emph{Darboux} or \emph{symplectic}.
	
	Using this coordinates we can build charts for the Lagrange Grassmannian $Lag(\Sigma)$. The map $S\mapsto V_S$ from the space of symmetric matrices to $Lag(\Sigma)$ maps onto the set of planes transversal to $\Pi$ (see for example \cite{bookgosson} for details).	
	
\subsection{Intersection index of Lagrangian subspaces}
\label{section: maslov index}
It is well known that the all pairs of transversal Lagrangian subspaces can be mapped to each other with a linear symplectomorphism~\cite[Theorem 1.15]{bookgosson}. This is no longer true for a triple of Lagrangian subspaces.

\begin{dfn}[Maslov index]
	Take three Lagrangian subspaces $L_0,L_1,L_2$. Consider the isotropic subspace $L_1' := L_1 \cap (L_0+L_2)$, if $l_1 \in L_1'$ then $l_1 = l_0+l_2$ with $l_i \in L_i$. The following quadratic form is called the Maslov form of the triple $(L_0,L_1,L_2)$:				
	\begin{equation*}
		m(l_1') = \sigma (l_0,l_2).
	\end{equation*}
By a slight abuse of notation we will also write $m(L_0,L_1,L_2)$ instead of just $m$ when we want to be explicit about which Lagrangian subspaces are used.	
	
	The numbers $\ind^+ m$, $\ind^- m$ and $sg \, m$ and $\dim \ker m$ are invariants of the triple $(L_0,L_1,L_2)$.
	The \textit{Kashiwara index} is the signature of the Maslov form:
	\begin{equation*}
		\tau(L_0,L_1,L_2) = sg \, m = \ind^+ m-\ind^- m.
	\end{equation*} 
	The \textit{negative Maslov index} is defined as 
	$$
	i\big (L_0,L_1,L_2\big) = \ind^- m.
	$$
\end{dfn}

\begin{exmp}
	Suppose $L_0$ and $L_2$ are transversal. We can identify the symplectic space with  the standard one $(\mathbb R^{2n}, \sigma)$ as given in \cref{eq: matrix J}.	
	Since any couple of Lagrangian subspaces can be mapped into each other, we can find a symplectomorphism which simultaneously maps $L_0$ to $B$ and $L_2$ to $\Pi$.
	
	Any $L_1$ can be represented as $L_1 = \{Aq+Cp=0, q \in B, p \in \Pi\}$ where $AC^* =C A^* $ and $rank[A,C] =n$.	If $A$ or $C$ is invertible then the matrix expression of the Maslov form is given by $-A^{-1}C$ or $C^{-1}A$ respectively, which have the same signature of $\mp AC^*$.			
\end{exmp}

The Kashiwara index and the Maslov index have the following properties:
\begin{itemize}
	\item \textbf{Alternating} $\tau(L_{s(0)},L_{s(1)},L_{s(2)}) =(-1)^{sg(s)}\tau(L_0,L_1,L_2)$, where $s$ is a permutation.
	\item \textbf{Cocycle property}~\cite[Theorem 1.32]{bookgosson}
	\begin{equation}
		\label{eq: cocycle kashiwara} \tau(L_0,L_1,L_2)-\tau(L_1,L_2,L_3)+ \tau(L_0,L_2,L_3)- \tau(L_0,L_1,L_3) =0.
	\end{equation}
	\item \textbf{Relation between the negative index} \cite[Lemma 5]{agrachev_quadratic_paper}\begin{equation}\label{eq: relation kashiwara maslov}\begin{aligned}
			\tau(L_0,L_1,L_2) &= -2\, i(L_0,L_1,L_2) +\dim(L_1 \cap (L_0+L_2))-\dim\ker m\\ 
			&=-2 \, i(L_0,L_1,L_2)+n-\dim(L_0\cap L_2)-\dim(L_0 \cap 	L_1)\\ & \quad-\dim(L_1 \cap L_2)+2\dim(L_0\cap L_1 \cap L_2).
		\end{aligned}
	\end{equation}
	\item \textbf{Symplectic reduction}
	If $V \subseteq L_0\cap L_2$ is an isotropic subspace we can consider $L^V:= (L\cap V^{\perp}+ V)/V$ which is a Lagrangian subspace of the reduced space. It holds that:
	\begin{equation}
		\label{eq: symplectic reduction maslov index}
		i\big(L_0,L_1,L_2\big) = i\big(L_0^V,L_1^V,L_2^V\big)
	\end{equation}
\end{itemize}

 The Maslov index satisfies only a \emph{generic} cocycle property. The next lemma will be used in the proof of \Cref{thm: comparison theorem} and shows that the defect of being a cocycle is measured by the intersections of the four Lagrangian subspaces one considers.
\begin{lemma}
	The following formulas hold:
	\begin{equation}
		\label{eq: coboundary index}
		\begin{split}
			\sum_{i=0}^{3} (-1)^{i+1}i \big(L_0,\dots,\hat L_i,\dots L_3\big) = \dim (L_1\cap L_3)-\dim(L_0\cap L_2)+\\+\sum_{i=0}^{3} (-1)^{i+1}\dim(L_0\cap\dots \cap\hat L_i \cap \dots \cap  L_3).
		\end{split}
	\end{equation}	
	\begin{equation}
		\label{eq: invariance cyclic permutation index}
		i\big(L_0,L_1,L_2\big) = i\big(L_1,L_2,L_0\big) = i\big(L_2,L_0,L_1\big).
	\end{equation}
	\begin{proof}
		The proof is just a computation using the cocycle identity for the signature in \cref{eq: cocycle kashiwara}. 
		
		We use formula \eqref{eq: relation kashiwara maslov} to get a relation between the signature and the index for each term on the left of~\eqref{eq: coboundary index}. The part involving intersections of pairs of Lagrangian spaces, after the summation, gives the contribution 
		\begin{equation*}
			2\Big (\dim(L_1\cap L_3)-\dim(L_0\cap L_2)\Big).
		\end{equation*}
		
		The triple intersection do not simplify and thus their coboundary appears. Sometimes it is useful to think about this remainder as the difference of the dimensions of two quotient spaces. $L_1\cap L_3$ in which we factor out the space $L_1\cap L_3 \cap L_0 + L_1 \cap L_3 \cap L_0$ and $L_0\cap L_2$ quotient by $L_0\cap L_2 \cap L_1 + L_0 \cap L_2 \cap L_3$.
		
		Next apply the first part with $L_3 = L_0$. The right hand side of the formula is then zero. Moreover the terms $i\big(L_0,L_1,L_0\big)$ and $i\big(L_0,L_0,L_2\big)$ are zero since the Maslov form is zero. 
		
		It follows that $i \big(L_0,L_1,L_2\big) = i\big(L_1,L_2,L_0\big)$, i.e. we obtain the invariance under cyclic permutations.
	\end{proof}
\end{lemma}

The following lemma will be used in the sequel.
\begin{lemma}
	\label{lemma: maslov annullatore}
Given the standard symplectic structure $(\Sigma,\sigma) = (T_\lambda( T^*M),ds_{\lambda})$ and three submanifolds $N_0,N_1,N_2 \subset M$. Assume that $\lambda \in A(N_0)\cap A(N_1)\cap A(N_2)$ and $N_1 \subseteq N_0$ (or  $N_1 \subseteq N_2$) , then following formula holds
$$
m\big( T_\lambda A(N_0), T_\lambda A(N_1), T_\lambda A(N_2)\big) \equiv 0.
$$	

Moreover if $M = M'\times M'$, $T^*M$ is endowed with the form $(-\sigma')	\oplus \sigma'$ and $A(N_i)$ are defined as in \cref{eq:ann_def}, the same is true.
	\begin{proof}
	Let $L_0 =  T_\lambda A(N_0)$, $L_1 =  T_\lambda A(N_1)$ and $L_2 =  T_\lambda A(N_2)$. Fix some coordinates in a neighbourhood of $\lambda$ such that $ds_\lambda$ is the standard form on $\mathbb R^{2n}\simeq T_{\lambda}(T^*M)$. The subspace $L_0+L_2$ is the space:\begin{equation*}
			\begin{pmatrix}
				\nu_0 \\ X_0
			\end{pmatrix} +\begin{pmatrix}
				\nu_2 \\ X_2
			\end{pmatrix},  \qquad X_i \in T_{\pi(\lambda)} N_i, \qquad \nu_i (T_{\pi(\lambda)} N_i) =0,
		\end{equation*}
for $i=0,2$. Since the sum above should lie in $L_1\cap (L_0 +L_2)$, we have that $X_0 + X_2 = X_1 \in T_{\pi(\lambda)}N_1$ and $\nu_0+\nu_2 =\nu_1 $, with $\nu_1$ such that $\nu_1 (T_{\pi(\lambda)} N_1) =0$. If we compute now the Maslov form, we get:
		\begin{equation*}
			\left\langle J \begin{pmatrix}
				\nu_0 \\ X_0
			\end{pmatrix} ,\begin{pmatrix}
				\nu_2 \\ X_2
			\end{pmatrix} \right\rangle =  
		\langle \nu_0,X_2\rangle-\langle \nu_2,X_0\rangle.
		\end{equation*}
		Suppose without loss of generality that $N_1 \subseteq N_0$. The equation $X_0+X_2= X_1$ implies that $X_2 = X_1-X_0 \in T_{\pi(\lambda)} N_0$ and thus $\langle \nu_0,X_2\rangle =0$. Therefore the quadratic form is the zero form since:
$$
\langle\nu_2 , X_0 \rangle =\langle\nu_2 , X_0+X_2 \rangle=\langle\nu_2+\nu_0 , X_0+X_2\rangle=\langle\nu_1 , X_1\rangle=0.
$$
For the second part, we work on the cotangent bundle of $M = M'\times M'$ which is isomorphic to $T^*M' \times T^*M'$. Label the coordinates as $(\lambda_0,\lambda_1)$, call the standard form on $T^*M'$, $\sigma'$ and consider the following diffeomorphism:
\begin{equation*}
	S: (\lambda_0,\lambda_1) \mapsto (-\lambda_0,\lambda_1).
\end{equation*}
It is straightforward to check that $S^*(\sigma'\oplus \sigma ') = (-\sigma')\oplus \sigma'$ and $S$ maps $A(N_i)$ as given in \cref{def: eq annullatore}  to the corresponding $A(N_i)$ as given by \cref{eq:ann_def}. Since Maslov index is invariant with respect to the action of symplectomorphisms, the statement follows.

	\end{proof} 
\end{lemma}

\subsection{Hermitian Maslov index}

\label{subsec:hermitian}

In the proof of Iteration Formulae it will be convenient to introduce complex coefficients and work on $\mathbb{C}^{2n}$. Maslov form extends to this setting in the obvious way. Consider the complex version of the symplectic form:
\begin{equation*}
	\sigma_\mathbb{C}(X,Y) = \sigma (\bar X, Y), \quad 	X,Y \in \mathbb{C}^{2n}.
\end{equation*}
If $V$ is Lagrangian as a real vector space, then $V \otimes \mathbb{C}$ is Lagrangian with respect to the new symplectic form. If we take three Lagrangian subspaces the Maslov form is still well defined. Suppose that $\lambda_1 \in L_1 \cap (L_0+L_2)$:
\begin{equation*}
	m(\lambda_1) = \sigma_\mathbb{C}(\lambda_0,\lambda_2) = \sigma (\bar\lambda_0,\lambda_2).
\end{equation*}

\begin{lemma}
	\label{lemma: hermitian maslov}
	The Maslov form is a Hermitian form.
	\begin{proof}
		Notice that $\overline{m(\lambda_1)} = \sigma (\lambda_0,\bar\lambda_2) = m(\bar \lambda_1)$. We have to show that $m(\lambda_1) =m( \bar \lambda_1)$ but this follows from the fact the subspaces are Lagrangian.
		\begin{equation*}
			\begin{split}
				m(\lambda_1) -m(\bar \lambda_1)&= \sigma (\bar \lambda_0,\lambda_2)-\sigma(\lambda_0,\bar\lambda_2) = \sigma(\bar \lambda_0+ \bar \lambda_2,\lambda_2)-\sigma(\lambda_0,\bar\lambda_2+\bar \lambda_0) \\ &= \sigma(\bar \lambda_1,\lambda_2)-\sigma(\lambda_0,\bar \lambda_1) = \sigma (\bar \lambda_1,\lambda_1) =0.
			\end{split}
		\end{equation*}
		This means that the quadratic form is real and thus $m$ is Hermitian.
	\end{proof}
\end{lemma}

Thus the eigenvalues of $m$ are real and the index and the signature is well defined exactly as in the real case. Here we list some of the properties of the $\sigma_\mathbb{C}$ and  complex Lagrange subspaces:
\begin{itemize}
	\item \textbf{Darboux basis} Since $\sigma_{\mathbb{C}}$ is non degenerate, every time two Lagrangian subspaces $L_0,L_1$ are considered, there exists a basis in which $\sigma_{\mathbb{C}}$ has the standard form.
	\item \textbf{Grassmannian of Lagrangian subspaces} 
	In the real case the Lagrange Grassmannian is a homogeneous space diffeomorphic to $U(n)/O(n)$. It turns out that the \emph{complex} one is diffeomorphic to $U(n)$ (and thus still real as a manifold). We can diagonalize the symplectic form obtaining:
	\begin{equation*}
		\frac{1}{2}\begin{pmatrix}
			i & -1 \\
			1 & -i
		\end{pmatrix} \begin{pmatrix}
			0 &-1 \\ 
			1 &0
	\end{pmatrix}\begin{pmatrix}
	-i &1 \\
	-1 &i
	\end{pmatrix} = \begin{pmatrix}
	i &0 \\ 0 & -i
	\end{pmatrix}.
	\end{equation*}
	Thus we have two subspaces on which $\sigma_\mathbb{C}$ is non degenerate, the eigenspace $V_i$ relative to $i$ and $V_{-i}$, the one relative to $-i$. It is thus clear that if $V$ is Lagrangian, $V$ must be transversal to both the eigenspaces. So it can always be represented as a graph of an invertible linear operator from $V_i \to V_{-i}$ (or vice versa). It remains to check what kind of linear maps are allowed. Using again the coordinates in which $\sigma_{\mathbb{C}}$ is diagonal we get:
	\begin{equation*}
		\sigma_{\mathbb{C}}\Big(\begin{pmatrix}
			x\\ R x 
		\end{pmatrix},\begin{pmatrix}
		y\\ R y 
	\end{pmatrix}\Big)= i \langle\bar{x},y \rangle -i \langle R^*\bar{R} \bar{x},y\rangle = i \langle (1-R^*\bar{R})\bar{x},y\rangle .
	\end{equation*} 
	Since we need this quantity to be zero for any $x, y \in \mathbb{C}^n$ we get $\bar{R}R^* =1$ and thus $R \in U(n)$.
	It follows that the complex Grassmannian is diffeomorphic to $U(n)$.
	\item \textbf{Atlas for the Lagrange Grassmannian} Take two transversal subspaces $L_0,L_1$. Using Darboux coordinates, we can build an affine chart as in the real case (compare with \cref{exmp:standard}). This time though we consider Hermitian matrices. The subspaces $V_S =\{(x,Sx): S = \bar S^*, x \in L_0\} $ are Lagrangian subspaces. 
	 
	\item \textbf{Properties of Kashiwara index} The proof of the cocycle property given in \cite[Theorem 1.32]{bookgosson} works in the Hermitian case as well since our quadratic forms are all real by \Cref{lemma: hermitian maslov}. It suffice to substitute the word \emph{symmetric} with the word \emph{Hermitian}. In particular all the properties listed in the previous section remain true in this setting, with real dimensions replaced by complex ones.

\end{itemize}

\section{Applications and proofs of \Cref{thm:discrete} and \Cref{thm:iteration}}
\label{section: proof applications}
\subsection{Reduction of the problem on a graph to a problem on an interval}
\label{subsec:reduction_to_1d}

Let us transform problem~\eqref{eq:graph_functional}-\eqref{eq:graph_boundary} to a problem on a single interval $[0,1]$. 
 We can assume without any loss of generality that $l_e = 1$ for every $e\in \cG_1$ simply by rescaling the time interval $[0,l_e]$, when $l_e < +\infty$. If $l_e =+ \infty$ we can compactify the semi-line $[0,+\infty)$ via a suitable change of coordinates in the time variable, provided that the compactification satisfies Assumption \ref{ass:legendre}, fact that may depend on the choice of change of variables. This will merely change the Lagrangians $\ell^e$ and  the vector fields $f^e_{t,u}$ and they can be redefined accordingly. 

Let us ignore for the moment the boundary conditions~\eqref{eq:graph_boundary} and treat the restriction of the optimal control problem to every edge as an independent problem. Define the map $\pi_e : M^{\cG_1} \to M_e$, it is the projection on the copy of $M$ relative to edge $e$. Then functional $J[u]$ and the control system can be seen as a part of an optimal control problem on $M^{\cG_1}$. We define the new Lagrangian $\ell: [0,1] \times M^{\cG_1} \times (\R^k)^{\cG_1}$ as
$$
\ell(t,q,u) = \sum_{e\in\cG_1} \ell^e(t,q_e,u_e)
$$
and a new family of time dependent vector fields $f^t_u \in Vec(M^{\cG_1})$ such that for a fixed edge $e\in \cG_1$
$$
(\pi_e)_*(f^t_u) = f^e_{t,u_e}.
$$

Finally, let us consider the boundary conditions~\eqref{eq:graph_boundary}. In order to make this construction cleaner, we need the following definitions.

\begin{definition}
Let $I = \{i_1,i_2,\dots,i_m\}$ be a finite set and $W$ an arbitrary set. A \textit{$I$-parametrized direct product of $W$} is
$$
W^{I} = W_{i_1} \times \dots \times W_{i_m},
$$
where $W_{i_j} = W$ for all $j=1,\dots,m$.
\end{definition}
Since we use $I$ as an index set, for a given subset $J\subset I$ we can define the projection map
$$
\pi_J: W^I \to W^J
$$ 
by forgetting the copies of $W$ indexed by the elements of $I\setminus J$.

\begin{definition}
Let $I,J$ be two finite index sets, $W$ an arbitrary set and $f: I \to J$ a surjective map. Then the \textit{pull-back product} $f^*(W^J)$ is a subset of $W^I$ characterized by the property that for every $j\in J$
$$
\pi_{f^{-1}(j)}(f^*(W^J)) = \left\{(q,q,\dots,q)\in W^{f^{-1}(j)}:q\in W_j \right\}.
$$
Similarly if $X\subset W^J$ the pull-back $f^*(X)$ is a subset of $W^I$ defined by the property that for every $j\in J$.
$$
\pi_{f^{-1}(j)}(f^*(X)) = \left\{(q,q,\dots,q)\in W^{f^{-1}(j)} \, : \, q\in \pi_j(X) \right\}. 
$$
\end{definition}

Let us look at an example when $I=\{1,2\}$ and $J=\{1\}$ and $f(1)=f(2) = 1$. In this case the preimage of $1$ is all of $I$ and hence
$$
f^*(W^J) = \pi_{f^{-1}(1)}(f^*(W^J)) = \{(q,q) \in W^I : q \in W\}, 
$$
which is just the diagonal. Similarly, if $X\subset W$ then $f^*(X)$ is the intersection of the diagonal with $X\times X \subset W \times W$.

Now we can describe the reduction procedure. The idea is intuitive, but a bit involved when written down formally. We want to use the orientations on edges of $\cG$ and pull-back the set of boundary constraints $N$, a priori defined just on the vertex set $\cG_0$, to a new manifold embedded in $ (M \times M)^{\cG_1}$. This allows us to separate the dynamic on each edge from the boundary conditions, exactly as we did a few lines above. 

Saying that each edge $e\in \cG_1$ is oriented is equivalent to having source and target maps $s,t:\cG_1 \to \cG_0$. The image of the source and the target of an edge $e$ are
$$
(q_e(0),q_e(1))  \in M_e \times M_e \simeq M \times M.
$$
Taking a product indexed by $\cG_1$ we obtain that
$$
(q(0),q(1)) \in  (M \times M)^{\cG_1} \simeq M^{\cG_1 \sqcup \cG_1},
$$
where $\cG_1 \sqcup \cG_1$ is a disjoint union of two copies of $\cG_1$. The source and the target maps induce a surjective map $s\sqcup t: \cG_1 \sqcup \cG_1 \to \cG_0$. This allows us to pull-back the boundary conditions to
$$
\tilde N = (s\sqcup t)^*(N).
$$
Hence we have reduced the optimal control problem~\eqref{eq:graph_functional}-\eqref{eq:graph_boundary} to an optimal control problem~\eqref{eq:control}-\eqref{eq:functional} with the configuration space $M^{\cG_1}$ and the boundary conditions $(s\sqcup t)^*(N)$, which encode all the information about the graph structure.

The final remark concerns the symplectic form on $(T^*M \times T^*M)^{\cG_1}$ that will be used for the calculation of the Maslov index. We will assume that each copy of $T^*M$ corresponding to a source vertex carries \emph{minus} the standard symplectic form $-\sigma$, while every target copy $T^*M$ carries $\sigma$, the standard one.
\subsection{Discretization}

\label{subsec:discretization}

In this subsection we prove \Cref{thm:discrete}. Fix a partition $\Xi = \{t_i:\,t_0=0,t_n =1, t_i<t_{i+1}\}$ of the unit interval. 

We will work with an optimal control problem and Assumptions \ref{ass:regularity}$-$\ref{ass:legendre}. In particular, Assumption \ref{ass:legendre}, the strong Legendre condition along the extremal, ensures that the Morse index is finite and that the conjugate points form a discrete set.  This will guarantee that, under mild conditions and after enough successive refinements of the partition, formula \eqref{eq: discretization formula} will give exactly the Morse index of the extremal.

Let us first prove the formula when only one extra vertex is introduced. Let $\gamma = \pi(\lambda)$ be an extremal curve in a problem with fixed end-points. Take a point $t^* \in (0,1)$.  Let us call $\gamma_1 = \gamma|_{[0,t^*]}$ and $\gamma_2 = \gamma |_{[t^*,1]}$ the restrictions. $Q_{\gamma_i}$ will denote the second variation of the segment as an extremal curve with fixed points. Recall that $\Pi_i = T_{\lambda(t_i)}(T^*_{\pi(\lambda(t_i))}M) \simeq T^*_{\pi(\lambda(t_i))}M$ is the vertical subspace over the point $\gamma(t_i)$.

\begin{prop}
	\label{prop: splitting curve}
	The index of the second variation $Q_\gamma$ can be computed as follows:
	\begin{equation}
		\label{eq: index splitting in two}
		\ind^-Q_\gamma = \ind^- Q_{\gamma_1}+\ind^- Q_{\gamma_2} + i\big (\Theta_2^{-1}(\Pi_2),\Pi_1,\Theta_1(\Pi_0)\big ) +k, 
	\end{equation}
	where $k = \dim(\Theta_2(\Pi_1)\cap \Pi_2)+\dim(\Theta_1(\Pi_0)\cap \Pi_1)-\dim(\Theta^{-1}_2(\Pi_2)\cap \Pi_1\cap \Theta_1(\Pi_0))$.
	\begin{proof}
		Let us consider the following three points in $M$:
		\begin{equation*}
			q_0 = \gamma(0), \quad q_1 = \gamma(t^*), \quad q_2 = \gamma(1).
		\end{equation*}
		
		Variations of $\gamma$ as a curve form $q_0$ to $q_2$ do not necessarily pass by the point $q_1$ at time $t^*$ but satisfy a continuity condition there. We perform the reduction to a single interval as discussed in Subsection~\ref{subsec:reduction_to_1d}. To do this we break up $[0,1]$ in two intervals and consider the dynamics separately (i.e. duplicate the variables). The new boundary conditions which allow us to glue the two pieces together are of the form:
		\begin{equation*}
			(\gamma_1(0),\gamma_2(t^*),\gamma_1(t^*),\gamma_2(1)) \in \{q_0\} \times \Delta \times \{q_2\} = \{(q_0,q_1,q_1,q_2) |\,q_1 \in M\}.
		\end{equation*}
		
		Now we are going to compare the following two problems. The first one is \emph{fixed end-points}, we impose that the curve starts from $(q_0,q_1)$ and arrives to $(q_1,q_2)$. The second one is curves satisfying the constraints given by the manifold $N = \{q_0\} \times \Delta \times \{q_2\}$ defined above.
		
	Recall that $\gamma$ is a projection of a solution $\lambda:[0,1]\to T^*M$ of the Hamiltonian system. Let us consider the tangent space to the annihilator of $N$ at the point $\underline{\lambda}={(\lambda(0),\lambda(1),\lambda(1),\lambda(2))}$. Fix a system of coordinates, which determines a complement to the subspace $\ker \pi_*= \Pi_1$ which we call $B$. In these coordinates the annihilator reads:
		\begin{equation*}
			T_{\underline{\lambda}} A(N) = \left\{\begin{pmatrix}
				\nu_1 \\
				\alpha + X \\
				\alpha + X \\
				\nu_2 
			\end{pmatrix}\,:\, \alpha, \nu_i \in \Pi_1, \, X \in B  \right\}.
		\end{equation*}
		
		The other space appearing is the graph of the two symplectomorphisms $\Theta_1$ and $\Theta_2$ coming from the Hamiltonian flows of PMP on intervals $[0,t^*]$ and $[t^*,1]$. It will be denoted by $\Gamma(\Theta_1 \times \Theta_2)$.
		
		Let us look at the subspace on which the Maslov form $m$ is defined, $	(T_{\underline{\lambda}} A(N) + \Pi^4) \cap \Gamma(\Theta_1 \times \Theta_2)$, where $\Pi^4 = \Pi_0 \times \Pi_1^2 \times \Pi_2$. This is defined by the following equations,
		\begin{equation*}
			\begin{pmatrix}
				\xi_1 \\ \xi_2 \\ \Theta_1(\xi_1) \\ \Theta_2(\xi_2)
			\end{pmatrix} = \begin{pmatrix}
				\nu_1 \\ \alpha + X \\\alpha +X \\ \nu_2
			\end{pmatrix}+\begin{pmatrix}
				\mu_1 \\ \mu_2 \\ \mu_3 \\ \mu_4
			\end{pmatrix} \iff \xi_1\in \Pi_0 , \Theta_2(\xi_2 ) \in \Pi_2,\, \xi_2-\Theta_1(\xi_1) \in \Pi_1,
		\end{equation*}
	where $\alpha$, $\nu_i$, for $i=1,2$ and $\mu_j$ for $j=1,\dots 4$ lie in the vertical subspace over the respective points, whereas $X \in B$ is in the horizontal space.
		In particular Maslov form reads:
		\begin{equation*}
			m(\xi_1,\xi_2) = \sigma(\mu_3-\mu_2,X) = \sigma (\Theta_1(\xi_1)-\xi_2,\xi_2) = \sigma (\Theta_1(\xi_1),\xi_2) = \sigma(\xi_2,-\Theta_1(\xi_1)). 
		\end{equation*}
		So, if we call $\eta = \Theta_2(\xi_2) \in \Pi_2$ and $\xi = \xi_1$ we have $\xi \in \Pi_0,\eta \in \Pi_2$ and $\Theta_2^{-1}(\eta)-\Theta_1(\xi) \in \Pi_1$ and see that the form coincides with:
		\begin{equation*}
			m (\Pi^4,\Gamma(\Theta_1 \times \Theta_2),T_{\underline{\lambda} }A(N)) = m (\Theta_2^{-1}(\Pi_2),\Pi_1,\Theta_1(\Pi_0)). 
		\end{equation*}
		
		The additional terms popping up in \Cref{thm: comparison theorem} are 
		$$
		\dim (\Gamma(\Theta_1 \times \Theta_2)\cap \Pi^4) = \dim(\Theta_1(\Pi_0) \cap \Pi_1)+\dim(\Theta_2^{-1}(\Pi_2) \cap \Pi_1)$$ 
		and 
		$$\dim(\Gamma(\Theta_1 \times \Theta_2)\cap \Pi^4\cap T_{\underline{\lambda} }A(N)) = \dim(\Theta_2^{-1}(\Pi_2)\cap \Theta_1(\Pi_0) \cap \Pi_1)$$
		 as a quick calculation shows.
	\end{proof}
\end{prop}

We can now prove \Cref{thm:discrete}.

	\begin{proof}[Prood of \Cref{thm:discrete}]
		The statement will follow from \Cref{prop: splitting curve}. First of all notice that in \cref{eq: index splitting in two} all terms are positive, this gives easily that $\ind^-Q \ge i\big(\Theta_2^{-1}(\Pi),\Pi, \Theta_1(\Pi)\big) $ when the partition is $\Xi = \{0,t^*,1\}$.
		For a general $\Xi$ apply \Cref{prop: splitting curve} iteratively to $\{0,t_{j-1},t_j\}$ where $j$ runs from $2$ to $n$. This allows to express the index of the second variation of $\gamma\vert_{[0,t_j]}$ as the sum of the index of the second variation of $\gamma\vert_{[0,t_{j-1}]}$ and $\gamma_j := \gamma\vert_{[t_{j-1},t_j]}$ plus Maslov index terms and dimensions of intersections. 
		
		Iteratively replacing the terms $\ind^- \, Q_{\gamma\vert_{[0,t_j]}}$ we obtain the following formula:
		\begin{equation*}
			\begin{split}
				\ind^-Q_{\gamma}= \sum_{j=0}^{n-1}\ind^- Q_{\gamma_j} + i \big(\Theta_{j+1,j}^{-1}(\Pi_{j+1}),\Pi_j,\Theta_{j,j-1} \circ \dots \circ  \Theta_{1,0}(\Pi_0)		\big)  \\
				+\dim(\Theta_{j,0}(\Pi_0) \cap \Pi_j) +\dim(\Theta_{j+1,j}(\Pi_j)\cap \Pi_{j+1})-\dim(\Theta_{j,0}(\Pi_0) \cap \Pi_j \cap\Theta^{-1}_{j+1,j}(\Pi_{j+1})).
			\end{split}
		\end{equation*}
	
		Here as described in the statement the maps $\Theta_{j,j-1}$ are the linearisation of the Hamiltonian flow of \cref{eq:hamilton}. The notation is related to law of composition of non autonomous flows and is justified by the fact that $\Theta_{j,k}\circ \Theta_{k,l} = \Theta_{j,l}$. 
		
		The index is presented as sum of three positive terms: the first one $Q_{\gamma_i}$ is zero when each segment $\gamma\vert_{[t_i,t_{i+1}]}$ is minimizing~\cite[Theorem 21.3]{bookcontrol}. Under Assumption~\ref{ass:legendre} this is the case when $\sup_i\vert t_i- t_{i-1}\vert$ is small enough (see \cite{bookcontrol} for instance). 
		The same goes for $\dim(\Theta_{i+1,i}(\Pi_i)\cap \Pi_{i+1}) -\dim(\Theta_{i,0}(\Pi_0) \cap \Pi_i \cap\Theta^{-1}_{i+1,i}(\Pi_{i+1}))$. 
		Moreover $\dim(\Theta_{i,0}(\Pi_0) \cap \Pi_i)$ is zero precisely  when $t_i$ is not a conjugate time for $\gamma$. 
		
		Thus equality holds exactly when our hypotheses on the partition are satisfied. 		
		\end{proof}

\begin{rmrk}
	The hypothesis on the partition $\Xi$ can be weakened if we change a bit our way of counting.  If we add to the dimension of the negative space the dimension of the null space of the Maslov form we can essentially forget about avoiding conjugate points of $\gamma$.
	
	You can see that the correction term $k$ in \Cref{prop: splitting curve} is in fact the dimension of the kernel of the Maslov form $m \big (\Theta_2^{-1}(\Pi_2),\Pi_1, \Theta_1(\Pi_1)\big )$.	
	The quantity $ \sum_{i=1}^{n-1} (\ind^-+\ker)\Big(m \big(\Theta_{i+1,i}^{-1}(\Pi_{i+1}),\Pi_i,\Theta_{i,i-1} \circ \dots \circ  \Theta_{1,0}(\Pi_0)
	\big) \Big) $ still approximates from below the negative index and includes the contribution of conjugate points of $\gamma$ that are possibly present in the partition.
\end{rmrk}

\begin{rmrk}
	If we combine \Cref{thm: comparison theorem} and \Cref{thm:discrete} we can obtain a formula for the index involving just the Maslov index $i$ and dimension of intersections for arbitrary boundary conditions. 
\end{rmrk}

\subsection{Filtration formula}
\label{subsec:filtration}

In the previous subsection we have proven a discretization formula for the fixed end-point problem on an interval. The idea was to introduce extra vertices inside the single edge and apply an iterative procedure of fixing each of the new vertices one by one. Note that if we would have fixed all of the vertices at the same time, a direct application of~\eqref{eq:main_index_formula} would result in computation of the Maslov index in a symplectic space of a very big dimension. Instead the recursive nature of the proof allows us to reduce greatly the dimensionality of the problem.

A way of reducing the dimensionality in formula~\eqref{eq:main_index_formula} for problems with separated boundary conditions is discussed in paper~\cite{baryshnikov}. The argument works when all of the Lagrangian spaces in the final formula are transversal. We can reproduce this argument in a greater generality using Theorem~\ref{thm: comparison theorem}.

Assume that each vertex $v\in\cG_0$ is constrained to lie on a separate submanifold $N_v \subset M$. We denote by $N$ boundary conditions, which are obtained after the reduction of the problem to an interval. We can introduce a filtration of vertices
$$
\emptyset = \cG_0^0 \subset \cG_0^1 \subset \dots \subset \cG_0^{| \cG_0|} = \cG_0,
$$ 
such that
$$
|\cG_0^j| = |\cG_0^{j-1}| + 1, \qquad i = 1,\dots,|\cG_0|.
$$
To each set $\cG_0^j$ we associate boundary conditions $N_j \subset N$ in the following way. We assume that vertices $v\in \cG_0^j$ vary on {$N_v$}, while vertices $v\in \cG_0 \setminus \cG_0^j$ are assumed to be fixed. Thus we activate variations of each individual vertex at a time and track how the index changes as we do so. 

We now apply Theorem~\ref{thm: comparison theorem} to compute 
$
\ind^- Q_{N_{j+1}} - \ind^- Q_{N_j}.
$
Let us introduce some simplifying notations. Recall that $s,t: \cG_1 \to \cG_0$ are the source and the target maps. Let $v_j \in \cG^{j+1}_0 \setminus \cG^j_0$ be the activated vertex. We introduce a separate notation for the set of edges that are incident to $v_j$:
$$
\cG^j_1 = s^{-1}(v_j) \, \cup\, t^{-1}(v_j).
$$
A naive guess would be that when we activate a vertex, the only relevant contributions come from the edges incident to a given vertex. Thus we define forgetful projections $\pi_{\cG^j_1}$ which forget all the edges except the ones incident to $v_j$:
$$
\pi_{\cG^j_1}: T_{\bar \lambda}(T^*M)^{\cG_1 }\times T_{\bar\lambda}(T^*M)^{\cG_1 } \to T_{\bar\lambda}(T^*M)^{\cG_1^j }\times T_{\bar\lambda}(T^*M)^{\cG_1^j }.
$$

Subspaces $T_{\bar\lambda}A(N_{j-1})$ and $T_{\bar\lambda}A(N_j)$ can have a big intersection. For sure this intersection contains the subset $V_j=\pi_{\cG^j_1}^{-1}(0) $, which is an isotropic subspace. This means that we can perform a symplectic reduction to the space $V_j^\perp /V_j$. Let 
$$
\pi_j: T_{\bar\lambda}(T^*M)^{|\cG_1|} \times T_{\bar\lambda}(T^*M)^{|\cG_1|} \to V_j^\perp /V_j  
$$
be the projection maps for each $j=1,\dots,|\cG_1|$. We can then define shortened notations for the images:
\begin{align*}
A_{j-1}^j  &= \pi_j(T_{\bar\lambda}A(N_{j-1})),\\ 
A_j^j  &= \pi_j(T_{\bar\lambda}A(N_j)),\\
\Gamma(\Theta_j) &= \pi_j(\Gamma(\Theta)).   
\end{align*}
By property \eqref{eq: symplectic reduction maslov index}, we can factor out $V_j$ in the definition of the Maslov index and get
$$
i(T_{\bar\lambda}A(N_{j-1}),\Gamma(\Theta),T_{\bar\lambda}A(N_j)) = i(A_{j-1}^j,\Gamma(\Theta_j),A_j^j)
$$
and for the same reason
$$
\dim \left(T_{\bar\lambda}A(N_{j-1})\cap\Gamma(\Theta)\right) - \dim \left(T_{\bar\lambda}A(N_{j})\cap\Gamma(\Theta)\cap T_{\bar\lambda}A(N_{j-1})\right) = \dim \left(A_{j-1}^j \cap\Gamma(\Theta_j)\right)  - \dim \left(A_{j-1}^j \cap\Gamma(\Theta_j) \cap A_j^j\right).
$$
Finally since $N_{j-1} \subset N_{j}$, we have that
$$
\dim (T_{\pi(\bar\lambda)}N_{j-1} \cap T_{\pi(\bar\lambda)}N_{j}) -\dim (T_{\pi(\bar\lambda)}N_{j-1})  = 0.
$$

Now we collect all of the terms and sum by the index $j=1,\dots,|\cG_0|$. As the result we obtain a formula that expresses the difference between the index of the second variation $Q$ of the original problem with the index of the second variation $Q_0:=Q_{N_0}$ of the problem with the same graph and fixed vertices:
\begin{align*}
\ind^- Q - \ind Q_0 &= \sum_{j=1}^{|\cG_0|}\ind^- Q_{N_j}- \ind^- Q_{N_{j-1}} = \\
&=\sum_{j=1}^{|\cG_0|} i(A_{j-1}^j,\Gamma(\Theta_j),A_j^j) + \dim \left(A_{j-1}^j \cap\Gamma(\Theta_j)\right)  - \dim \left(A_{j-1}^j \cap\Gamma(\Theta_j) \cap A_j^j\right).
\end{align*}
This is the same formula as in~\cite{baryshnikov} modulo terms containing dimensions of intersections.

We end this subsection with a couple of remarks regarding this formula. First of all, in practice the dimensions are reduced even more because $A^j_{j-1} \cap A_j^j \neq \emptyset$. Nevertheless further reductions depend on the structure of the graph and the filtration chosen. Secondly, at first sight it may seem that the formula is a sum of local contributions, because we have used only edges incident to a given vertex in the derivation. However, this is not the case. The non-locality is hidden in the reduced space $V^\perp_j / V_j$ and the corresponding projection $\pi_j$. 

For example, in the case when $\cG_1$ is a tree, we can define a partial order $\leq$ on $\cG_0$ by saying that $v \leq w$ if the minimal path from $v$ to the root crosses less or equal number of vertices than the minimal path from $w$ to the root. If we choose a filtration, which orders vertices one by one compatible with the partial ordering, one can identify the set $\cG_0^j$ with a sub-tree of $\cG$. Then the formula for the indices $i(A_{j-1}^j,\Gamma(\Theta_j),A_j^j)$ will contain terms involving symplectomorphisms of all of the edges in the sub-tree determined by $\cG_0^j$ and not only of the incident edges. This can be checked via a long but straightforward calculation.

\subsection{Iteration Formulae}
\label{subsec:iteration}
Now we prove \Cref{thm:iteration}. For clarity we prove separately the two formulas since the strategies in the two cases are a bit different.

Suppose that $\gamma$ is a closed periodic extremal trajectory. It is straightforward to see that iterations (i.e. concatenation of $\gamma$ with itself) are still critical points.

We will use the following notation: $\gamma^k$ will denote the $k-$th iteration of $\gamma$ whereas $\ind^-Q_\gamma$ and  $\ind^-Q_{\gamma^k}$ the Morse index of $\gamma$ and $\gamma^k$ respectively as periodic trajectories. We want to compute the difference $\ind^-Q_{\gamma^k}-k\, \ind^-Q_{\gamma}$.

First of all we compute the difference $\ind^- Q_{\gamma^{k}}-\ind^- Q_{\gamma^{k-1}}$. Let us consider the following manifolds of constraints:
\begin{equation*}
	\begin{split}
		\Delta^\circlearrowright := \{ (q_1,q_2,q_2,q_1) : q_i\in M, \} \subset M^2 \times M^2,\\
		\Delta^2 = \{(q_1,q_2,q_1,q_2) : q_i \in M\} \subset M^2 \times M^2.
	\end{split}
\end{equation*}

When we restrict to variations satisfying the boundary conditions given by $\Delta^\circlearrowright$, we consider variations of $\gamma^k$ as a periodic trajectory, whereas when we take boundary conditions $\Delta^2$, we consider independent variations of $\gamma^{k-1}$ and $\gamma$ as periodic trajectories. See~\Cref{fig:iteration} for a visual explanation between the two boundary conditions. Now we prove the following lemma:

\begin{lemma}
	\label{lemma: index k-1 interate}
	Here $n = \dim(M)$, let $\Gamma^j = \Gamma(\Theta^j)$ the graph of $\Theta^j$. Then:
	\begin{equation*}
		\ind^- Q_{\gamma^k}-  \ind^-Q_{\gamma^{k-1}}= \ind^-Q_{ \gamma} + i \big (\Gamma^{k-1},T_{\underline{\lambda}}A(\Delta),\Gamma^k)-n  +\dim\big(\ker(\Theta^{k-1}-1) \big).			
	\end{equation*}
	\begin{proof}
		The statements follows applying \Cref{thm: comparison theorem}. We take as $N_1 = \Delta^\circlearrowright$ and as $N_2 = \Delta^2$. The part coming from the dimension is immediate, the intersection of the tangent spaces has dimension $n$ while the dimension of $\Delta^2$ is $2n$. So we get a $-n$.
		
		For the part concerning the intersection between annihilators and graphs, one can check that $T_{(\lambda(0),\lambda(0))}A(\Delta^2)\cap \Gamma(\Theta^{k-1} \times \Theta)$ is isomorphic to the sum of $\ker(\Theta^{k-1}-1)$ and $\ker(\Theta-1)$. The triple intersection consists of  $\ker(\Theta-1)$ and thus the term in the statement. 
		
		From the definitions it follows that, when we impose the boundary conditions $\Delta^2$, we have $\ind^- Q_{N_2}= \ind^-Q_{ \gamma^{k-1}}+\ind^- Q_\gamma$, so the only thing to check is the Maslov index part.
		
		The equation defining the subspace are the following:
		\begin{equation*}
			\begin{pmatrix}
				\xi_1 \\ \xi_2 \\ \Theta^{k-1}(\xi_1) \\ \Theta(\xi_2)
			\end{pmatrix} = \begin{pmatrix}
				X_1+Y_1 \\ X_2+Y_2 \\ X_2+Y_1 \\X_1+Y_2
			\end{pmatrix} \quad X_i,Y_i,\xi_i \in T_{\lambda(0)}(T^*M).
		\end{equation*}
		
		By subtracting the second and the third equations, and then the first and the fourth equations we find
		$$\xi_2-\Theta^{k-1}(\xi_1) = Y_2-Y_1,$$ 
		$$ \Theta^{k-1}(\xi_1)-\Theta^{k}(\xi_2) = \Theta^{k-1}(Y_1-Y_2).$$ 
		Changing coordinates and setting $\eta = \Theta^{k-1}(\xi_1)$, $Y_1-Y_2 = \eta_1$ and $\eta_2 = \xi_2$ we get:
		\begin{equation*}
			(\eta,\eta) \in T_{\underline{\lambda}}A(\Delta)  \cap (\Gamma^{k-1}+\Gamma^k)\iff  \begin{pmatrix}
				\eta \\\eta
			\end{pmatrix} = \begin{pmatrix}
				\eta_1+\eta_2\\ \Theta^{k-1}(\eta_1)+\Theta^k(\eta_2) 
			\end{pmatrix}
		\end{equation*}
		So we can see that the Maslov form reduces to a form on $\Delta \cap (\Gamma^{k-1}+\Gamma^k)$.		
		Moreover the quadratic form reads:
		\begin{equation*}
			\begin{split}
				m(\xi_1,\xi_2)=\sigma(X_2,Y_1-Y_2)-\sigma (X_1,Y_1-Y_2) = \sigma(\xi_2-\Theta(\xi_2),Y_1-Y_2)\\
				= \sigma(\eta_2-\Theta(\eta_2),\eta_1) \\
				=\sigma(\eta_2,\eta_1)-\sigma(\Theta(\eta_2),\eta_1) \\
				= -\sigma(\eta_1,\eta_2)+\sigma(\Theta^{k-1}(\eta_1),\Theta^k(\eta_2)). 
			\end{split}
		\end{equation*}
		Which is exactly $m(\Gamma^{k-1},T_{\underline{\lambda}}A(\Delta),\Gamma^k)$ in the coordinates just introduced. And thus the formula is proved.
	\end{proof}
\end{lemma}

 The first iteration formula is now a direct consequence of the Lemma just proved:

\begin{thm*}[Iteration Formulae I]
	The index of the $k-$th iteration of $\gamma$ as a periodic trajectory satisfies:
	\begin{equation}
		\label{eq: iteration formula II}
		\ind^-Q_{\gamma^k} - k\, \ind^-Q_\gamma= \sum_{j=1}^k i\big(\Gamma(\Theta^{j-1}),T_{\underline{\lambda}}A(\Delta),\Gamma(\Theta^{j})\big)-\dim(M) + \dim(\ker(\Theta^{j-1}-1)).
	\end{equation} 
	\begin{proof}
		We will use an inductive procedure in a similar spirit as in the proof of \Cref{thm:discrete}. First we will look at $\gamma^k$ as the concatenation of $\gamma^{k-1}$ and $\gamma$ and express the difference $\ind^-Q_{\gamma^k}-\ind^-Q_\gamma$ in terms of $\ind^-Q_{\gamma^{k-1}}$.

		This is the first step of the scheme and was proved in \Cref{lemma: index k-1 interate}. Then we apply the argument iteratively to obtain:
		\begin{equation*}
			\ind^- Q_{\gamma^k} -k \, \ind^- Q_\gamma = \sum_{j=1}^k i(\Gamma^{j-1},T_{\underline{\lambda}}A(\Delta),\Gamma^{j})-n + \dim(\ker(\Theta^{j-1}-1)).
		\end{equation*}
		Which is precisely the formula in the statement.
	\end{proof}
\end{thm*}

Now we prove the second iteration formula.

\begin{thm*}[Iteration Formulae II]
	The index of the $k-$th iteration of $\gamma$ as a periodic trajectory satisfies:
	\begin{equation}
		\label{eq: iteration formula I}
				\ind^-Q_{\gamma^k} - k\, \ind^-Q_\gamma= \sum_{j=1}^{k-1} \dim(M)-\dim(\ker (\Theta-\omega^j))-i (\Gamma(\Theta),\Delta,\Gamma(\omega^j\Theta)).
	\end{equation} 
	Where $\omega$ is a primitive $k-$th root of the unity.  
	\begin{proof}
		We work on the space $M^k = M\times \dots \times M$. The first set of boundary conditions we are going to consider is the following:
		\begin{equation*}
			\Delta^\circlearrowright := \{ (q_1,\dots,q_k,r_1\dots r_k) : r_i,q_i\in M, q_i = r_{i-1} \} \subset M^k \times M^k.
		\end{equation*} 
		
		Set $q_0 = \gamma(0) = \gamma(1)$. Any curve satisfying the boundary conditions $\Delta^\circlearrowright$ at $(q_0, \dots , q_0)$ gives a variation of the $k-$th iterate of $\gamma$ seen as periodic trajectory. 
		
		The other sets of constraints we are going to introduce are the following:
		\begin{equation*}
			\begin{aligned}
			\Delta^{k} &= \{(q_1,\dots,q_k,q_1,\dots, q_k) : q_i \in M\} \subset M^k\times M^k,\\
				\underline{q}_0 &= \{(q_0,\dots,q_0) : q_0=\gamma(0) = \gamma(1)\}.
		\end{aligned}
		\end{equation*}
		
		The first boundary condition is the product of $2k$ copies of the diagonal. Any curve satisfying this set of constraints at point $(q_0, \dots , q_0)$ is a variation of $\gamma^k$ as $k$ independent periodic trajectories $\gamma$. The second boundary condition corresponds to $k$ copies of a single point $q_0$. Variations of $\gamma^k$ satisfying these latter conditions are $k$ independent variations of $\gamma$ as a trajectory with fixed points. 
		
		Set for brevity $\Delta^k = T_{\underline{\lambda}}A(\Delta^k)$, $\Delta^\circlearrowright = T_{\underline{\lambda}}A(\Delta^\circlearrowright)$ and 
	$
	\Gamma = \Gamma(\Theta\times\dots \times \Theta)
	$ 
	to be the product of $k$ copies of $\Gamma(\Theta)$. We have $T_{\underline{\lambda}}A(\underline{q}_0) = \Pi_{\lambda(0)}^{2k}= \Pi^{2k}$ where $\lambda(0)$ is the initial covector of the lift to the cotangent bundle.	
		
		First of all we compute directly $\ind^-Q_{\gamma^k}$ using \Cref{thm: comparison theorem}, comparing with the fixed points problem. We get:
		\begin{equation*}
			\ind^-Q_{\gamma^{k}} = k \, \ind^-{Q_0} +  i\big(\Pi^{2k},\Gamma,\Delta^\circlearrowright \big) +\dim(\Gamma \cap \Pi^{2k})-\dim(\Gamma \cap \Pi^{2k} \cap A(\Delta^\circlearrowright)).
		\end{equation*}
		
		Here the notation $\ind^- {Q_0}$ stands for the index of $Q$ at $\gamma$ seen as a trajectory with fixed end points. First of all we analyse the term $i\big(\Pi^{2k},\Gamma,\Delta^\circlearrowright \big)$. To compute it we present the Maslov form as the direct sum of $k$ forms defined on a $n-$dimensional subspace. This is done in \Cref{lemma: diagonalizzazione maslov}, where we use the \emph{complexified} version of Maslov index. 
		
		The term $i\big(\Pi^{2k},\Gamma,\Delta^\circlearrowright \big)$ is thus the sum of contributions of the type $i \big(\Pi^2,\Gamma(\omega^j \Theta),\Delta\big)$ where $\omega$ is a primitive root of unity. 
		
		\begin{equation*}
			i\big(\Pi^{2k},\Gamma,\Delta^\circlearrowright \big) = \sum_{j=0}^{k-1} i \big(\Pi^2,\Gamma(\omega^j \Theta),\Delta\big).
		\end{equation*}
		Now we apply \Cref{thm: comparison theorem} to the second set of boundary conditions, i.e. $\Delta^k$. We find that:
		\begin{equation*}
			k \ind^- Q_\gamma = k \ind^- {Q_0} + i \big(\Pi^{2k},\Gamma,\Delta^k \big)+ \dim(\Gamma\cap \Pi^{2k})- \dim(\Gamma\cap \Pi^{2k}\cap \Delta^k).
		\end{equation*}
		Exactly as in the previous case the piece $i\big(\Pi^{2k},\Gamma,\Delta^k\big)$ splits as a sum. But this time the reason is more apparent: we are considering independent variation on each iteration.
		It follows that $i\big(\Pi^{2k},\Gamma,\Delta^k\big) = k \,  i \big (\Pi^2,\Gamma(\Theta),\Delta\big)$
	
		Now we subtract the two equations and we are left with the following expression for  $\ind^-Q_{\gamma^k} - k\, \ind^-Q_\gamma$:
		\begin{equation*}
			 \ind^-Q_{\gamma^k} - k\, \ind^-Q_\gamma = \sum_{j=0}^{k-1}\Big( i \big (\Pi^2, \Gamma(\omega^j \Theta),\Delta^\circlearrowright \big)-i \big (\Pi^2, \Gamma(\Theta),\Delta \big)\Big) +\dim (\Gamma \cap \Pi^{2k} \cap \Delta^k)- \dim (\Gamma \cap \Pi^{2k} \cap \Delta^{\circlearrowright}).
		\end{equation*}
		
		 Let's rewrite the term involving the intersections. It is straightforward to see that $\dim(\Gamma \cap \Pi^{2k} \cap \Delta^k) = k \dim (\Gamma(\Theta)\cap \Pi^2\cap \Delta)$. In turn this can be easily seen to be $k \dim (\ker(\Theta-1)\cap \Pi)$.
		 
		 For the second piece it holds that:
		 \begin{equation*}
		 	\begin{split}
		 	\dim (\Gamma\cap \Pi^k\cap \Delta^\circlearrowright) &= \sum_{j=0}^{k-1} \dim (\ker(\Theta-\omega^j)\cap \Pi).\\\end{split}
		 \end{equation*}
	 	We prove this below, in \Cref{prop: computation maslov iteration}. Putting all together we get:
	 	\begin{equation}
	 		\label{eq: proof iteration II triple intersection}
	 		\dim(\Gamma \cap \Pi^{2k} \cap \Delta^k)-\dim (\Gamma\cap \Pi^k\cap \Delta^\circlearrowright) = \sum_{j=0}^{k-1} \dim (\Gamma(\Theta)\cap \Pi^2\cap \Delta)- \dim (\ker(\Theta-\omega^j)\cap \Pi).
 		\end{equation}
	 	
		Now we can use the cocycle property given in \cref{eq: coboundary index} with the subspaces $\Pi, \Gamma(\omega^j \Theta), \Gamma(\Theta)$ and $\Delta$ to  express the terms in the sum using the Maslov index of the spaces $\Gamma(\omega^j \Theta)$ and  $\Delta$. These computations are collected in \Cref{prop: computation maslov iteration}. What we find is that:
		\begin{equation*}
			\begin{split}			
		i\big(\Pi^2, \Gamma(\omega^j \Theta),\Delta \big)-i \big( \Pi^2, \Gamma( \Theta),\Delta \big) &= -i \big(\Gamma(\Theta),\Delta,\Gamma(\omega^j\Theta) \big) +\dim (M) \\& \,\, -\dim (\ker(\Theta -1) \cap\Pi) +\dim (\ker(\Theta -\omega^{-j}) \cap\Pi)-\dim \ker(\Theta-\omega^{-j}).		
	\end{split}
	\end{equation*}

	Since we are summing over $j=0,\dots,k-1$ and $\omega$ is a primitive $k$-th root of unity, we have that $\sum_{j =0}^{k-1}\dim (\ker(\Theta -\omega^{-j}) \cap\Pi) = \sum_{j =0}^{k-1}\dim (\ker(\Theta -\omega^{j}) \cap\Pi) $ and thus the intersection of the eigenspaces with the fibre cancel out with the part coming from triple intersection given in \cref{eq: proof iteration II triple intersection}.
		
	Summing up we finally obtain:
	\begin{equation*}
		\ind^-Q_{\gamma^k} - k\, \ind^-Q_\gamma = \sum_{j=1}^{k-1}\dim (M) -\dim \ker(\Theta-\omega^j) - i \big ( \Gamma(\Theta),\Delta,\Gamma(\omega^j \Theta) \big).
	\end{equation*}
	Which is exactly the statement of the theorem. 
	\end{proof}
\end{thm*}

		\begin{lemma}
			\label{lemma: diagonalizzazione maslov}
			Let $\omega \in \mathbb{C}$ be a primitive $k-$th root of the unity. The Maslov form $m(\Pi^{2k},\Gamma,\Delta^\circlearrowright) = \bigoplus_{i=0}^{k-1} m_i$ where:
			\begin{equation*}
				m_i = m(\Pi^2,\Gamma(\omega^i \Theta),\Delta).
			\end{equation*} 
			\begin{proof}
				We will use the Hermitian version of Maslov index. Any real subspace $V$ appearing in the proof will stand for its complexification $V \otimes \mathbb{C}$ without any mention of the tensor product operation. Let us write down the equation defining the space $(\Pi^{2k} + \Delta^\circlearrowright)\cap \Gamma$.
				\begin{equation*}
					v \in \Gamma \iff v = (\xi_1, \dots,\xi_k,\Theta(\xi_1),\dots,\Theta(\xi_k) ), \quad  \xi_j \in T_{\lambda_0}(T_{q_0}^*M).
				\end{equation*} 
				
				On the other hand belonging to $\Pi^{2k}+\Delta^\circlearrowright$ means:
				\begin{equation*}
					v \in \Pi^{2k}+\Delta^\circlearrowright \iff v = (\mu_1, \dots, \mu_k,\nu_1, \dots,\nu_k), \quad \mu_{i+1}-\nu_{i} \in \Pi,
				\end{equation*}
				where $\mu_{k+1}=\mu_1$. So the space $(\Pi^{2k} + \Delta^\circlearrowright)\cap \Gamma$ is given by $\{(\xi_1,\dots, \xi_k)\, : \, \xi_{i+1}-\Theta(\xi_{i}) \in \Pi\}$.
				
				Maslov form is computed in the following way. Let
\begin{align*}
\xi_i &= X_i+\alpha_i, \\
\Theta(\xi_i) &= X_{i+1}+\beta_i, 
\end{align*}
where $\alpha_i,\beta_i \in \Pi$, $X_i \in T_{\lambda_0}(T^*M)$. Then we have				
				\begin{equation*}
					\begin{split}
						 m(\underline{\xi}) = \sum_{i=1}^{k} -\sigma (\bar \alpha_i,X_i)+\sigma(\bar \beta_i,X_{i+1}) =\sum_{i=1}^{k} -\sigma (\bar \alpha_i,X_i)+\sigma( \bar \beta_{i-1},X_i)=\sum_{i=1}^{k} \sigma (-\bar \alpha_i+\bar \beta_{i-1},X_i) \\
						=\sum_{i=1}^{k} \sigma (-\bar \alpha_i+\bar \beta_{i-1},\xi_i)=\sum_{i=1}^{k} \sigma (\Theta(\bar \xi_{i-1})-\bar \xi_i,\xi_i) = 	\sum_{i=1}^{k}  \sigma (\Theta(\bar \xi_{i-1}), \xi_i)-\sigma(\bar \xi_i,\xi_i). 
					\end{split}				 
				\end{equation*}
				Where in the third equality we simply shifted the second index cyclically.

				Suppose that $\omega$ is a primitive $k-$th root of the identity and make the following change of variables.
				\begin{equation*}
					\underline{\xi} = (\xi_1,\dots,\xi_k)\mapsto \left(\sum_{i=1}^{k}\xi_i,\dots , \sum_{i=1}^k\omega^{j(i-1)}\xi_i,\dots,\sum_{i=1}^k\omega^{(k-1)(i-1)}\xi_i \right) =:\underline{\eta},
				\end{equation*}
				which essentially is just the Kronecker product of the identity with the transpose of Vandermonde's matrix obtained with $\{1,\omega,\dots,\omega^{k-1}\}$.  In the new coordinates the equation reads:
				\begin{equation*}
					\begin{split}
						\eta_l-\omega^{l-1}\Theta(\eta_l) &= \sum_{i=1}^k \omega^{(l-1)(i-1)}\xi_i-\sum_{i=1}^{k}\omega^{(l-1)i}\Theta(\xi_{i}) \\
						&= \sum_{i=1}^k \omega^{(l-1)i}\xi_{i+1}-\omega^{(l-1)i}\Theta(\xi_{i}) \\&=\sum_{i=1}^k \omega^{(l-1)i}(\xi_{i+1}-\Theta(\xi_{i}))
						\in \Pi.
					\end{split}
				\end{equation*}
So in the new coordinates the space  $(\Pi^{2k} + \Delta^\circlearrowright)\cap \Gamma$ splits as the direct sum $\bigoplus_{l=1}^k \{\eta \, : \, \eta-\omega^l\Theta(\eta) \in \Pi\}$.

The inverse transformation is given by the following rule:
				\begin{equation*}
					\xi_i = \frac{1}{k} \sum_{l=1}^k \omega^{-(i-1)(l-1)} \eta_{l}.
				\end{equation*}
		
	If we plug in the second term of the Maslov form we have:
				\begin{equation*}
					\begin{split}
						\sum_{i=1}^k \sigma(\Theta(\overline{\xi_{i}}),\xi_{i+1}) =\frac{1}{k^2} \sum_{i,l,s=1}^{k} \sigma(\Theta(\omega^{(i-1)(s-1)}\bar \eta_s),\omega^{-i(l-1)}\eta_l) \\
						=\frac{1}{k^2}\sum_{i,l,s=1}^{k}\omega^{i(s-l)}\omega^{-(s-1)} \sigma(\Theta(\bar \eta_s),\eta_l) =\frac{1}{k^2} \sum_{l,s=1}^k\Big (\sum_{i=1}^k(\omega^{i(s-l)})\omega^{1-s}\sigma(\Theta(\bar \eta_s),\eta_l)\Big). 
					\end{split}
				\end{equation*}
				
				In particular the only non zero terms are those for which $s = l$ since the sum of powers of any primitive root (up to $k$) is zero. 
				
				We can handle similarly the first term. In this way we find that the Maslov form on our subspace splits in the following way:
				\begin{equation*}
					m(\underline{\eta}) = \frac{1}{k} \sum_{s=1}^k \sigma(\overline{\omega^{s-1}\Theta( \eta_s)},\eta_s)-\sigma(\bar\eta_s,\eta_s).
				\end{equation*}
				The factor $\frac{1}{k}$ is irrelevant for us and comes just from the change of coordinates. The last step is to identify the summands with $m(\Pi^2,\Gamma(\omega^{s-1} \Theta),\Delta)$. Let's write down the kernel for these forms. The space we have to look at is $(\Pi^2+\Delta)\cap \Gamma(\omega^{s-1}\Theta)$. It is defined by:
				\begin{equation*}
					\eta = \alpha + X \quad \omega^{s-1}\Theta(\eta) = \beta+X \quad \alpha, \beta\in \Pi.
				\end{equation*}
			By the definition the Maslov form is given by
			\begin{equation*}
				m(\eta) = -\sigma(\bar \alpha,X)+\sigma(\bar \beta,X) = \sigma(\overline{\omega^{s-1}\Theta(\eta)}-\bar \eta,\eta). 
			\end{equation*}
			\end{proof}
		\end{lemma}
		
		\begin{prop}
			\label{prop: computation maslov iteration}
			The following relation holds:
			\begin{equation}
				i\big(\Pi^2, \Gamma(\omega^j \Theta),\Delta \big)-i \big( \Pi^2, \Gamma(\Theta),\Delta \big) = -i \big(\Gamma(\Theta),\Delta,\Gamma(\omega^j\Theta) \big) +\dim (M) +d_j,
			\end{equation}
			where $d_j = -\dim (\ker(\Theta -1) \cap\Pi) +\dim (\ker(\Theta -\omega^{-j}) \cap\Pi)-\dim \ker(\Theta-\omega^{-j})$.
			Moreover the space $\Gamma \cap \Pi^{2k}\cap \Delta^\circlearrowright $ splits as a direct sum and its dimension is given by:
			\begin{equation*}
				\dim (\Gamma \cap \Pi^{2k}\cap \Delta^\circlearrowright ) = \sum_{j=0}^{k-1} \dim (\ker(\Theta-\omega^j)\cap \Pi).
			\end{equation*}	
			\begin{proof}
				The second part can be deduced by the proof of \Cref{lemma: diagonalizzazione maslov}. In fact the space $\Pi^{2k}\cap \Gamma\cap \Delta^\circlearrowright$ is isomorphic to $\bigoplus_i\ker(\Theta-\omega^i)\cap \Pi$. This can be either directly computed from the definition of the spaces or deduced in the following way.
				
				Let $P$ represent the standard $k-$cycle which maps $\xi_i \to \xi_{i+1}$ and $\xi_k\to \xi_1$. 
				A direct calculation shows that $\Delta^\circlearrowright = \Gamma(P)$. Thus any element of $\Delta^\circlearrowright \cap \Gamma$ can be written as 
				$$
				\begin{pmatrix}
					\xi \\ P\xi 
				\end{pmatrix}=
			\begin{pmatrix}
				\xi \\ diag(\Theta)(\xi)
			\end{pmatrix} \iff  P^{-1}diag(\Theta)(\xi) = \xi \iff diag(\Theta)P^{-1}(\eta) = \eta, \text{ where } \eta = P\xi.
				$$	
				 i.e. an eigenvalue problem.
				
				 The core of the proof of \Cref{lemma: diagonalizzazione maslov} consisted in the diagonalization of the following matrix:
				\begin{equation*}
					\begin{pmatrix}
						\Theta &  &\\
						&\ddots & \\
						& &\Theta
					\end{pmatrix}P^{-1} \sim 	\begin{pmatrix}
						\omega^0\Theta &  &\\
						&\ddots & \\
						& &\omega^{k-1}\Theta
					\end{pmatrix}
				\end{equation*} 
	with the remaining elements understood to be zero. The transformation diagonalizing the matrix we used preserves the fibre. So it follows that $\Pi^{2k}\cap \Gamma\cap \Delta^\circlearrowright$ is the sum of the eigenspaces $\ker(\Theta-\omega^j)$ intersected with the fibre $\Pi$.
				
	Now we prove the first part of the proposition. Let us apply the cocycle property to $\Pi^2,\Gamma(\omega^j\Theta),\Gamma(\Theta) $ and $\Delta$.
				\begin{equation*}
					\begin{split}
						i&\big (\Pi^2,\Gamma(\omega^j\Theta),\Delta \big)-i \big(\Pi^2,\Gamma(\Theta),\Delta \big) = i \big (\Gamma(\Theta),\Pi^2,\Gamma(\omega^j \Theta)\big) \\&\qquad   -i \big(\Gamma(\Theta),\Delta,\Gamma(\omega^j \Theta)\big)+ c_i,\\
						c_i &= \dim(\Theta(\Pi)\cap \Pi)-\dim(\ker(\Theta-1)\cap \Pi) +\dim(\ker(\Theta-\omega^{-j})\cap \Pi)+\\&\qquad  - \dim\ker(\Theta-\omega^{-j}).
					\end{split}
				\end{equation*}
				
			The formula is almost the one given in the statement except for the terms $\dim(\Theta(\Pi)\cap \Pi)$ and $i\big(\Gamma(\Theta),\Pi^2,\Gamma(\omega^j \Theta)\big)$ and a lacking $\dim(M)$.
				
			We can compute the Maslov index term in the following way. Notice that $\Gamma(\omega^j \Theta)$ and $\Gamma(\omega^l \Theta)$ are transversal if the index $j$ is different form $l$. It follows that the space on which the form is defined is $\Pi^2$. Moreover the equations are $\xi_1+\xi_2 = \nu_1 \in \Pi$ and $\Theta(\xi_1+\omega^j \xi_2) = \nu_2 \in \Pi $. Thus Maslov form reads:
				\begin{equation*}
					m(\nu_1,\nu_2) =-\sigma (\bar \xi_1,\xi_2)+\omega^j\sigma(\Theta(\bar\xi_1),\Theta(\xi_2)) = (\omega^j-1)\sigma(\bar \xi_1,\xi_2) . 
				\end{equation*}
				
				We can invert the equations to write them on $\Pi^2$. We get $\xi_2 = \frac{1}{1-\omega^j}(\nu_1-\Theta^{-1}(\nu_2))$ and $\xi_1 = \frac{1}{1-\omega^j}(\Theta^{-1}(\nu_2)-\omega^j \nu_1)$ and thus the form is equivalent to:
				\begin{equation*}
	m(\nu_1,\nu_2) = \frac{1}{\bar\omega^j-1}(\sigma(\bar\nu_2,\Theta(\nu_1))+\omega^{-j}\sigma(\bar\nu_1,\Theta^{-1}(\nu_2))). 
				\end{equation*}
				
				This form has zero signature and kernel isomorphic to two copies of $\Theta(\Pi)\cap \Pi$. This is a general fact and can be seen as follow. Suppose the matrix representing the quadratic form has the following expression:
				\begin{equation*}
					\cM = \begin{pmatrix}
						0 & X \\ \bar X^* &0
					\end{pmatrix} \quad m(\nu_1,\nu_2) =  \langle \bar\nu_1,X \nu_2 \rangle +\langle \bar \nu_2,\bar X^* \nu_1 \rangle  .
				\end{equation*}
				
				Let be $Q$ and $R$ unitary matrices which gives the singular values decomposition for $X$, i.e. $ QXR = D$ for $D=diag(d^2_i)$, diagonal and with non negative entries.
				
				Apply the following change of coordinates to $\cM$:
				\begin{equation*}
					\begin{pmatrix}
						Q & 0 \\ 0 & \bar R^*
					\end{pmatrix} \begin{pmatrix}
						0 & X \\ \bar X^* &0
					\end{pmatrix}\begin{pmatrix}
						\bar Q^* & 0 \\ 0 &  R
					\end{pmatrix} = \begin{pmatrix}
						0 & D \\ D &0
					\end{pmatrix}.
				\end{equation*}
				
				And then apply another change:
				\begin{equation*}
					\begin{pmatrix}
						1 & -1 \\ 1 & 1
					\end{pmatrix} \begin{pmatrix}
						0 & D \\ D &0
					\end{pmatrix}\begin{pmatrix}
						1 & 1 \\ -1 &  1
					\end{pmatrix} = \begin{pmatrix}
						-2D & 0 \\ 0 &2D
					\end{pmatrix}.
				\end{equation*}
		
		Thus the non zero eigenvalues of the matrix $\cM$ are $\pm d^2_i$, where $d^2_i>0$ are the positive singular values of $X$. The kernel of $\cM$ has dimension $2 \dim \ker(X)$.

		This is precisely our situation: fix a Lagrangian complement to the fibre $\Pi$, and consider the matrices associated to $\Theta$ and $ \frac{\omega^{-j}}{\omega^{-j}-1}J\Theta^{-1}$. In blocks they can be written as:
				\begin{equation*}
					\Theta = \begin{pmatrix}
						A &B \\ C &D
					\end{pmatrix}, \quad  J\Theta^{-1} = \begin{pmatrix}
						C^* &-A^* \\ D^* &-B^*
					\end{pmatrix}, \quad J \Theta = \begin{pmatrix}
					-C &-D \\ A &B
				\end{pmatrix}.
			\end{equation*}
		We are using coordinates in which the fibre $\Pi$ is the span of the first $n$ coordinates. Thus the block we have to consider is always the upper left one. Our form, with this conventions, is written as:
		\begin{equation*}
				m(\nu_1,\nu_2) = \left \langle \bar \nu_2, \frac{1}{1-\omega^{-j}}C\nu_1\right\rangle  + \left\langle \bar \nu_1, \frac{\omega^{-j}}{\omega^{-j}-1}C^*\nu_2\right\rangle. 
		\end{equation*}
So for us $X = \frac{\omega^{-j}}{1-\omega^{-j}}C^*$. 
				Thus our form has zero signature, is defined on a $2 \dim (M)$ dimensional vector space and the kernel is isomorphic to two copies of the kernel of $X$. The latter is easily seen to be  $\Theta(\Pi) \cap \Pi$.
				
				Thus it follows that $i\big(\Gamma(\Theta),\Pi^2,\Gamma(\omega^j \Theta)\big) = \dim(M)-\dim(\Theta(\Pi)\cap \Pi) $. Inserting above we get the formula in the statement. 
			\end{proof}
		\end{prop}
		
		We can consider the function $\mathbb S^1 \ni z \mapsto i (\Gamma(\Theta),\Delta,\Gamma(z\Theta))$. It has very nice properties and an explicit description in terms of the monodromy matrix $\Theta$. These ideas are collected in the following proposition.
		
		\begin{prop}
			\label{lemma: properties function on circle}
			The number $i\left(\Gamma(\Theta),\Delta,\Gamma(\omega^j\Theta)\right)$ corresponds to the number of negative eigenvalues of the following matrix:
			\begin{equation*}
				M_{\omega^j} =\frac{1}{1- \omega^{-j}}J\Big(\omega^{-j}+1-\omega^{-j}\Theta-\Theta^{-1}\Big). 
			\end{equation*}
			If we consider the function $\mathbb{S}^1 \ni z \mapsto i\big (\Gamma(\Theta),\Delta,\Gamma(z\Theta)\big )$, it is locally constant with at most $2n$ jumps at eigenvalues of $\Theta$. Moreover the jumps are bounded in amplitude by $\dim(\ker(\Theta-z))$ where $z \in \mathbb S^1$.
			\begin{proof}
				The first part is just a straightforward computation. Take for any $\alpha \in \mathbb S^{1}$:
				\begin{equation*}
					\begin{pmatrix}
						\xi_1 \\ \Theta(\xi_1)
					\end{pmatrix} +\begin{pmatrix}
						\xi_2 \\ \alpha \Theta(\xi_2)
					\end{pmatrix}
					=\begin{pmatrix}
						X \\ X
					\end{pmatrix} \Rightarrow \begin{cases}
						(1-\alpha)\xi_2 = X-\Theta^{-1}(X), \\
						(\alpha-1)\xi_1 = \alpha X-\Theta^{-1}(X).
					\end{cases}
				\end{equation*} 
				
				If $\alpha \ne 1$ the two graphs are always transversal and the Maslov quadratic form can be written in terms of the variable $X$:
				\begin{equation*}
					\begin{split}
						m(X) &=- \sigma(\bar\xi_1,\xi_2) +\alpha \sigma(\Theta(\bar\xi_1),\Theta(\xi_2))=(\alpha-1)\sigma(\bar\xi_1,\xi_2)  \\ &=\frac{1}{1-\bar \alpha}\sigma\Big(\bar \alpha\bar X- \Theta^{-1}(\bar X),X-\Theta^{-1}(X)\Big) \\&= \frac{\bar \alpha +1}{1-\bar{\alpha} } \sigma(\bar X,X)-\frac{1}{1-\bar\alpha}\sigma((\bar \alpha\Theta+\Theta^{-1})(\bar X),X). 
					\end{split}
				\end{equation*}
				
				It follows that the kernel is $M_\alpha = \frac{1}{1-\bar \alpha}J\Big(\bar \alpha+1-\bar\alpha\Theta-\Theta^{-1}\Big) $.

				For the second part notice that the map $\alpha \mapsto M_\alpha$ is continuous away from $1$ with values in the space of Hermitian matrices. 
				
				A change of index can occur only at those points in which the determinant of $M_z$ is zero, thus at most $2n$ times. Moreover the jumps are the following:
				\begin{equation*}
					\det(M_\alpha) = 0 \iff \det(\bar \alpha +1 -\bar \alpha \Theta-\Theta^{-1})=0,\, \alpha \ne 1.
				\end{equation*}
				
				In particular, notice that $\Theta$ and $\Theta^{-1}$ can be put in the same block triangular form. For example one can choose to put $\Theta$ in its Jordan form. On the diagonal, at a block corresponding to eigenvalue $\lambda$ of $\Theta$, the elements are $\bar \alpha + 1 -\bar \alpha \lambda-\frac{1}{\lambda}$. This quantity is zero if and only if $\alpha = \frac{\lambda}{\vert \lambda \vert^2}$ i.e. if $\alpha$ is an eigenvalue of $\Theta$ that lies on the circle.
				
				Thus the jumps are at most $2n$. The part on the bound follows by this observation: take a Jordan block of $\Theta$ with eigenvalue $\lambda$. Then the corresponding block of $\Theta^{-1}$ will have $\bar \lambda$ on the diagonal and $(-1)^k\bar{\lambda}^{k+1}$	on the $k-$th upper diagonal. This implies that on the first upper diagonal of the $\lambda$ block of $\bar \lambda \Theta +\Theta^{-1}$ we considered you end up with $-\bar \lambda + \bar{\lambda}^2$, which is different from zero. Thus each $\lambda-$block contributes with a single eigenvalue and so the jumps are controlled by $\dim(\ker(\Theta-\lambda))$. 	
			\end{proof}
		\end{prop}

\section{Index formulas and proof of \Cref{thm: comparison theorem}}
\label{section: proof comparison}

\subsection{Reduction to a variational problem with fixed end-points}
\label{subsec:reduction_to_fixed}
In this section we prove \Cref{thm: comparison theorem}. We will consider first the case in which the boundary constraints in \eqref{eq:boundary} are separated. This means that we look for a minimizer $\gamma$ with initial point $\gamma(0)$ in $N_0$ and final point $\gamma(1)$ in $N_1$, where $N_0,N_1 \subset M$ are embedded submanifolds. The general case will be reduced to this one.

Given an extremal trajectory $\gamma$ of the optimal control problem \eqref{eq:control}-\eqref{eq:functional} with $N = N_0 \times N_1$, we will now construct an extended optimal control problem with fixed end-points and interpret $\gamma$ as the restriction of an extremal $\hat \gamma$. Moreover we will show that the two problems are equivalent on a neighbourhood of $\gamma$ and $\hat\gamma$.

Denote $q_i = {\gamma}(i) \in N_i$ for $i = 0,1$. We want to construct the extended control system in such a way that all admissible curves $\hat{\alpha}$ connecting $q_0$ with $q_1$ sufficiently close to $\gamma$ are concatenations $\hat \alpha = \alpha_1 * \alpha * \alpha_0$, where $\alpha_i$ are curves inside $N_i$ and $\alpha$ connects $N_0$ to $N_1$. 

To do so fix neighbourhoods $\cO(q_i)\subset M$, $i = 0,1$ of the two points $q_0,q_1$. If $\cO(q_i)$ are sufficiently small, we can construct regular foliations of $\cO(q_i)$ such that $N_i \cap \cO(q_i)$ are leaves passing through $q_i$. We can consider the union of all tangent space to each leaf of the foliation. This gives us two integrable distributions $D_0$ and $D_1$ in each neighbourhood. Without loss of generality, we can choose a set of vector fields $f_i^j$, $j = 1,\dots, \dim N_i$ defined on $\cO(q_i)$, which generate these distributions:  
\begin{equation*}
	\text{span}\left \{f_i^j(q),\, j = 1,\dots, \dim N_i \right \} = D_i(q), \quad \forall q \in \mathcal{O}(q_i), \quad i = 0,1.
\end{equation*}
Using these vector fields we define an extended control system for times from $-1$ to $2$, which is linear in the controls for $t \in [0,1]^c$. Namely:
\begin{equation*}
	\hat{f}^t_u(q) = \begin{cases}
		f_0(q)u_0  := \sum_{j=1}^{\dim N_0} f_0^j(q)u_{0j}, \quad &\text{if }t <0,\\ 
		f_u^t(q), &\text{if }t \in [0,1],\\
		f_1(q)u_1 := \sum_{j=1}^{\dim N_1} f_1^j(q)u_{1j}, &\text{if }t>1.
	\end{cases}
\end{equation*}
where $u_0 \in \R^{\dim N_0}$ and $u_1 \in \R^{\dim N_1}$. The space of the extended controls $(u_0,u,u_1)$ is isomorphic to $\hat{\cU}=\mathbb{R}^{\dim N_0} \oplus L^\infty([0,1],\R^k) \oplus \mathbb{R}^{\dim N_1}$ and it is identified with functions which are locally constant on $[-1,0]$ and on $[1,2]$ with values in $\mathbb{R}^{\dim N_0}$ and $\mathbb{R}^{\dim N_1}$ respectively.

\begin{figure}[ht]
	
	\begin{center}
		\tikzset{every picture/.style={line width=0.75pt}} 
		
		\begin{tikzpicture}[x=0.75pt,y=0.75pt,yscale=-0.75,xscale=0.75]
			
			\draw  [fill={rgb, 255:red, 238; green, 237; blue, 237 }  ,fill opacity=1 ] (477,134) .. controls (493,94.95) and (587,114) .. (567,134) .. controls (547,154) and (558,203.95) .. (578,233.95) .. controls (598,263.95) and (508,275.95) .. (488,245.95) .. controls (468,215.95) and (461,173.05) .. (477,134) -- cycle ;
			\draw  [fill={rgb, 255:red, 212; green, 212; blue, 212 }  ,fill opacity=1 ] (450.67,116) .. controls (466.67,76.95) and (560.67,96) .. (540.67,116) .. controls (520.67,136) and (531.67,185.95) .. (551.67,215.95) .. controls (571.67,245.95) and (481.67,257.95) .. (461.67,227.95) .. controls (441.67,197.95) and (434.67,155.05) .. (450.67,116) -- cycle ;
			\draw  [fill={rgb, 255:red, 255; green, 255; blue, 255 }  ,fill opacity=1 ] (427,100) .. controls (443,60.95) and (537,80) .. (517,100) .. controls (497,120) and (508,169.95) .. (528,199.95) .. controls (548,229.95) and (458,241.95) .. (438,211.95) .. controls (418,181.95) and (411,139.05) .. (427,100) -- cycle ;
			\draw  [fill={rgb, 255:red, 238; green, 237; blue, 237 }  ,fill opacity=1 ] (174.06,123) .. controls (158.06,83.95) and (64.06,103) .. (84.06,123) .. controls (104.06,143) and (93.06,192.95) .. (73.06,222.95) .. controls (53.06,252.95) and (143.06,264.95) .. (163.06,234.95) .. controls (183.06,204.95) and (190.06,162.05) .. (174.06,123) -- cycle ;
			\draw  [fill={rgb, 255:red, 212; green, 212; blue, 212 }  ,fill opacity=1 ] (200.06,106) .. controls (184.06,66.95) and (90.06,86) .. (110.06,106) .. controls (130.06,126) and (119.06,175.95) .. (99.06,205.95) .. controls (79.06,235.95) and (169.06,247.95) .. (189.06,217.95) .. controls (209.06,187.95) and (216.06,145.05) .. (200.06,106) -- cycle ;
			\draw  [fill={rgb, 255:red, 255; green, 255; blue, 255 }  ,fill opacity=1 ] (226.06,89) .. controls (210.06,49.95) and (116.06,69) .. (136.06,89) .. controls (156.06,109) and (145.06,158.95) .. (125.06,188.95) .. controls (105.06,218.95) and (195.06,230.95) .. (215.06,200.95) .. controls (235.06,170.95) and (242.06,128.05) .. (226.06,89) -- cycle ;
			\draw [line width=1.5]    (185,179) .. controls (225,149) and (425,220.95) .. (465,190.95) ;
			\draw    (197.06,109) .. controls (237.06,79) and (419,139.95) .. (477.06,120.95) ;
			\draw    (185,179) .. controls (198,156.95) and (204,130.95) .. (197.06,109) ;
			\draw    (465,190.95) .. controls (457,161.95) and (472,139.95) .. (477.06,120.95) ;
			\draw    (320.36,184.52) -- (327.53,185.53) ;
			\draw [shift={(330.5,185.95)}, rotate = 188.04] [fill={rgb, 255:red, 0; green, 0; blue, 0 }  ][line width=0.08]  [draw opacity=0] (10.72,-5.15) -- (0,0) -- (10.72,5.15) -- (7.12,0) -- cycle    ;
			\draw    (198.36,146.52) -- (199.4,139.99) ;
			\draw [shift={(199.86,137.02)}, rotate = 458.97] [fill={rgb, 255:red, 0; green, 0; blue, 0 }  ][line width=0.08]  [draw opacity=0] (10.72,-5.15) -- (0,0) -- (10.72,5.15) -- (7.12,0) -- cycle    ;
			\draw    (325.64,110.22) -- (333.34,111.42) ;
			\draw [shift={(336.3,111.89)}, rotate = 188.88] [fill={rgb, 255:red, 0; green, 0; blue, 0 }  ][line width=0.08]  [draw opacity=0] (10.72,-5.15) -- (0,0) -- (10.72,5.15) -- (7.12,0) -- cycle    ;
			\draw    (464.71,157.8) -- (463.66,163.34) ;
			\draw [shift={(463.1,166.29)}, rotate = 280.72] [fill={rgb, 255:red, 0; green, 0; blue, 0 }  ][line width=0.08]  [draw opacity=0] (10.72,-5.15) -- (0,0) -- (10.72,5.15) -- (7.12,0) -- cycle    ;
			\draw  [fill={rgb, 255:red, 255; green, 255; blue, 255 }  ,fill opacity=1 ] (183.5,178) .. controls (183.5,176.58) and (184.62,175.43) .. (186,175.43) .. controls (187.38,175.43) and (188.5,176.58) .. (188.5,178) .. controls (188.5,179.42) and (187.38,180.57) .. (186,180.57) .. controls (184.62,180.57) and (183.5,179.42) .. (183.5,178) -- cycle ;
			\draw  [fill={rgb, 255:red, 255; green, 255; blue, 255 }  ,fill opacity=1 ] (462.5,190.95) .. controls (462.5,189.53) and (463.62,188.38) .. (465,188.38) .. controls (466.38,188.38) and (467.5,189.53) .. (467.5,190.95) .. controls (467.5,192.38) and (466.38,193.53) .. (465,193.53) .. controls (463.62,193.53) and (462.5,192.38) .. (462.5,190.95) -- cycle ;
			
			\draw (317.33,192.13) node [anchor=north west][inner sep=0.75pt]    {$\gamma $};
			\draw (326.33,86.46) node [anchor=north west][inner sep=0.75pt]    {$\hat{\alpha }$};
			\draw (148,78.79) node [anchor=north west][inner sep=0.75pt]    {$N_{0}$};
			\draw (484,87.46) node [anchor=north west][inner sep=0.75pt]    {$N_{1}$};
			\draw (171.67,180.13) node [anchor=north west][inner sep=0.75pt]    {$q_{0}$};
			\draw (470,192.13) node [anchor=north west][inner sep=0.75pt]    {$q_{1}$};

		\end{tikzpicture}
	\end{center}
	\caption{An admissible extended variation $\hat{\alpha}$ of an extremal curve $\gamma$ \label{pic:extend}}
\end{figure}
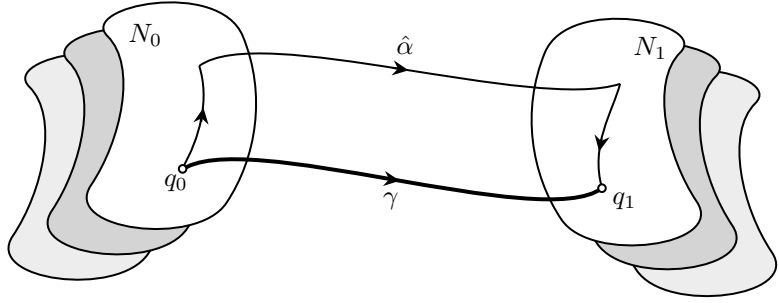

In Figure~\ref{pic:extend} the construction is explained visually. The local foliations are depicted in grey scale with the white surfaces being $N_0$ and $N_1$. An admissible curve $\hat{\alpha}$ of an extended system is confined to the leaf of the starting point up to time $0$, evolves with the law prescribed by the initial system and then continues inside the leaf reached at time $1$. In particular, if we are interested in extended admissible curves $\hat{\alpha}$ which connect $q_0$ and $q_1$, we get that $\hat{\alpha}|_{[0,1]}$ must connect $N_0$ to $N_1$.

We define the new optimal control problem as
\begin{equation}
	\label{eq:opt_ext1}
	\dot{q}=\hat{f}^t_{\hat{u}(t)}(q),  \qquad \hat u \in \R^{\dim N_0} \oplus L^\infty([0,1],\R^k) \oplus \R^{\dim N_1},
\end{equation}
\begin{equation}
	\label{eq:opt_ext2}
	q(-1)= q_0, \;\; q(2) = q_1,
\end{equation}
\begin{equation}
	\label{eq:opt_ext3}
	\min_{\hat u \in \hat{\cU}}\hat\varphi(\hat u) = \min_{\hat u \in \hat{\cU}} \int_0^1 \ell(t,u(t),q_u(t)) dt.
\end{equation}
where $u = \hat u|_{[0,1]}$.

\begin{lemma}
	\label{lem:equivalence}
	Optimal control problems \eqref{eq:control}-\eqref{eq:functional} with $N = N_0 \times N_1$ and~\eqref{eq:opt_ext1}-\eqref{eq:opt_ext3} are locally equivalent.
\end{lemma}

\begin{proof}
	If an admissible curve $\alpha$ of the original control system with control $u$ satisfies $\alpha(i) \in \cO(q_i)\cap N_i$ for $i = 0,1$, then it can be lifted to an admissible curve of the new system connecting $q_0$ and $q_1$. Indeed, take the unique controls $u_i$ for which
	$$
	\exp\left( f_0(\cdot )u_0 \right) q_0 = \alpha(0), \qquad \exp\left( f_1(\cdot )u_1 \right) \alpha(1) = q_1, 
	$$
	where $\exp$ denotes the flow of the vector field inside the brackets at time $t=1$. Then the lift $\hat{\alpha}$ is defined as
	
	\begin{equation*}
		\hat{\alpha}(t) =\begin{cases}
			\exp\left(t f_0(\cdot )u_0 \right) q_0, \quad & \text{if }t <0, \\ 
			\alpha(t), &\text{if }t \in[0,1], \\
			\exp\left(t f_1(\cdot )u_1 \right) \alpha(1), &\text{if }t>1,
		\end{cases} 
		\qquad 
		\hat{u}(t) =\begin{cases}
			u_0, \quad & \text{if }t <0, \\ 
			u(t), &\text{if }t \in[0,1], \\
			u_1, &\text{if }t>1.
		\end{cases}
	\end{equation*}
	
	Conversely, if we have an admissible curve $\hat \alpha$ such that $\hat \alpha(-1)=q_0$ and $\hat{\alpha}(2)=q_1$, then its restriction $\alpha = \hat{\alpha}|_{[0,1]}$ is a curve that connects $N_0$ to $N_1$ and such that $\alpha(i) \in \cO(q_i)$. Thus we obtain a local bijection between two spaces of admissible curves. 
	
	In order to finish the proof, it only remains to notice that $\hat{\varphi}(\hat{u}) = \varphi(u)$.
\end{proof}

\subsection{Computation of first and second variation of the extended problem}
\label{subsec:variation_formulas}

Let $\lambda:[0,1]\to T^*M$ be an extremal satisfying PMP for the problem \eqref{eq:control}-\eqref{eq:functional} with $N=N_0 \times N_1$. Denote by $\tilde{u}$ the corresponding control and by $\gamma(t) = \pi(\lambda(t))$, $t\in [0,1]$ the extremal curve on the manifold $M$. As was indicated in the proof of Lemma~\ref{lem:equivalence} we can extend $\gamma$ to an admissible curve of~\eqref{eq:opt_ext1}-\eqref{eq:opt_ext3} as

$$
\hat\gamma = \begin{cases}
	q_0, \quad & \text{if }t <0, \\ 
	\gamma(t), &\text{if }t \in[0,1], \\
	q_1, &\text{if }t>1,
\end{cases}
\qquad
\hat{\tilde u}(t) =\begin{cases}
	0, \quad & \text{if }t <0, \\ 
	\tilde u(t), &\text{if }t \in[0,1], \\
	0, &\text{if }t>1.
\end{cases}
$$
In order to simplify slightly the notations, we will omit in the future the hat symbol for $\hat{\tilde{u}}$ by essentially identifying $\tilde{u} \sim (0,\tilde{u},0)$.

In this section we compute the first and second variations of the problem~\eqref{eq:opt_ext1}-\eqref{eq:opt_ext3} at a critical point $\tilde{u}$. To do this, we use the already existing formulas for the fixed end-point problem which can be found in several references such as \cite{bookcontrol}. In order to do this we need to define all of the objects appearing in the formulas for the first and second variation which are written below.

The Hamiltonian of PMP is given by:
\begin{equation*}
	\hat h_{\hat u}^t(\lambda) = \begin{cases}
		\langle \lambda , f_0 u_0 \rangle =: h_0u_0, \quad & \text{if }t <0,\\
		\langle \lambda , f_u^t(\pi(\lambda))\rangle -\ell(t,u,\pi(\lambda)), &\text{if }t \in [0,1],\\
		\langle \lambda , f_1 u_1 \rangle =: h_1u_1,  &\text{if }t>1.
	\end{cases}
\end{equation*}
According to PMP a minimal control must maximize $\hat h_{\hat u}^t(\lambda)$. Given an extremal $\hat{\lambda}:[-1,2]\to T^*M$, let $\lambda_0$ and $\lambda_1$ be its restrictions to the intervals $[-1,2]$ and $[0,1]$. Since this family is linear in $u_0$ and $u_1$, $\lambda_i$ must lie in the annihilators $A(N_i)$. In particular, if $\lambda: [0,1]\to T^* M$ was an extremal of the original problem such that $\gamma(t) = \pi(\lambda(t))$, $\forall t\in [0,1]$, we can extend $\lambda$ to an extremal $\hat \lambda$ of problem~\eqref{eq:opt_ext1}-\eqref{eq:opt_ext3} exactly as before:
$$
\hat\lambda(t) = \begin{cases}
	\lambda(0), \quad & \text{if }t <0, \\ 
	\lambda(t), &\text{if }t \in[0,1], \\
	\lambda(1), &\text{if }t>1.
\end{cases}
$$

Denote by $\hat{\Phi}_t$ the flow generated by $\hat h_{\tilde u}^t$, or more precisely

\begin{equation*}
	\hat \Phi_t = \begin{cases}
		\mathbb{I}, &\text{if }t<0,\\
		\Phi_t, & \text{if }t \in [0,1], \\
		\Phi_1, & \text{if }t>1 .	
	\end{cases}
\end{equation*}
Since the vector field associated to $\tilde{u}$ is zero on $[-1,0] \cup [1,2]$ the flow is a constant transformation. Composing the Hamiltonian with the flow $\hat \Phi_t$ gives us

\begin{equation*}
	\hat b_{\hat u}^t(\lambda) = \big (\hat h_{\hat u}-\hat h_{\tilde{u}(t)} \big ) \circ \hat \Phi_t(\lambda) = \begin{cases}
		\langle \lambda , f_0 u_0 \rangle, \quad & \text{if }t <0,\\
		b_u^t(\lambda),  &\text{if }t \in [0,1],\\
		\langle \, \cdot \, , f_1 u_1 \rangle \circ \hat \Phi_1(\lambda), &\text{if }t>1.
	\end{cases} 
\end{equation*}
Then we can define
\begin{equation*}
	\hat Z_t = \partial_{\hat u} \overrightarrow{\hat b}_{\hat u}^t|_{\lambda = \lambda(0)}.
\end{equation*}
and denote $Z_0$ and $Z_1$ to be restrictions of $\hat{Z}_t$ to the time intervals $[-1,0]$ and $[1,2]$ correspondingly, so that 
\begin{equation*}
	\hat Z_t = \begin{cases}
		Z_0, \quad & \text{if }t <0,\\
		Z_t, & t \in [0,1], \\ 
		Z_1, \quad & \text{if }t >1.\\
	\end{cases}
\end{equation*}
It is worth noting that $Z_0$ and $Z_1$ are constant, since $b^t_{\hat{u}}(\lambda)$ is linear in $u_0$ and $u_1$ for $t\in[0,1]^c$. Finally we can define the quadratic form
$$
H_t = \left.\frac{\p^2 \hat{b}^t_{\hat{u}}}{\p \hat u^2}\right|_{\hat{u} = \tilde{u}}.
$$
Note that since $b^t_{\hat{u}}$ is linear in the control parameters for $t\in [-1,0] \cup [1,2]$, we have $H_t \equiv 0$ on the two intervals.

Recall that $\Pi := \Pi_{\lambda_0}$ denotes the vertical subspace, namely the tangent space to fibre $T^*_qM$ described in \cref{eq: fibre definition}.
With the notation set above, the kernel of the differential of the endpoint mapping and the second variation are:
\begin{equation}
	\label{differential endpoint mapping}
	\ker d_{\tilde{u}} E = \left\{\hat v \in \mathbb{R}^{\dim N_0} \oplus L^{\infty}([0,1],\R^k)\oplus \mathbb{R}^{\dim N_1} : \int_{-1}^2 \hat Z_t \hat v(t)dt\in \Pi\right \},  
\end{equation}
\begin{equation}
	\label{formula second variation}
	Q(\hat v,\hat w) = \int_{-1}^{2} \big [- H_t(\hat v(t),\hat w(t))-\int_{-1}^t\sigma (\hat Z_{\tau}\hat v(\tau),\hat Z_t \hat w(t)) d \tau \big ] dt,
\end{equation}
where $\hat v, \hat w \in \ker d_{\tilde{u}} E $. We can expand the expressions for the first and second variations knowing the particular form of $\hat Z_t$. We split the integrals into three integrals over the intervals $[-1,0]$, $[0,1]$ and $[1,2]$ and simplify the integrands using the skew-symmetry of $\sigma$:

They read as:

\begin{equation}
	\label{eq:differential endpoint mapping}
	\ker d_{\tilde{u}} E = \left\{v \in L^{\infty}[0,1], v_i \in \mathbb{R}^{\dim N_i} :\int_0^1 Z_tv(t)dt + Z_0v_0+Z_1v_1 \in \Pi\right\},
\end{equation}

\begin{equation}
	\label{quadratic form}
	Q(\hat v,\hat w) = \int_0^1\Big [-H_t(v(t),w(t))-\sigma \left(Z_0v_0 + \int_0^tZ_\tau v(\tau),Z_t w(t)\right)d \tau \Big ]dt -\sigma \left(Z_0v_0+\int_0^1 Z_t v(t) dt,Z_1w_1\right),
\end{equation}
where we have used the fact that $\hat Z_t$ is constant for $t\in [0,1]^c$ and its image lies in a Lagrangian subspace, and hence $\sigma(\hat Z_t \hat v(t),\hat Z_\tau \hat w(\tau)) = 0$ for all $\tau,t \in [-1,0]$, all $\tau,t \in [1,2]$ and any variations $\hat v,\hat w \in \ker d_{\tilde{u}} E$.

We finish the discussion of the first and second variations with an important observation concerning $Z_iv_i$, $i=0,1$.

\begin{lemma}
	\label{lemma:boundary_manifolds}
	$(P^i_{\tilde u})_*Z_iv_i$ are tangent to $A(N_i)$ for $i = 0,1$.
\end{lemma}

\begin{proof}
	By PMP the initial (and final) covector annihilates $N_0$ (resp. $N_1$). 
	
	Recall that $f_{0i}$ generate the tangent space to $N_0$ close to $q_0$. We define the Hamiltonians
	$$
	l_i(\lambda) = \langle \lambda, f_{0i} \rangle, \qquad i= 1,\dots,\dim N_0.
	$$
	Then $A(N_0)$ can be equivalently described as the common part of the zero locus of $l_i$:
	\begin{equation*}
		A(N_0) = \{\lambda \in T^* M : \pi(\lambda) \in N_0, \,l_i(\lambda) = 0, \,i=1,\dots,\dim N_0\}.
	\end{equation*}
	But then by the definition of a Hamiltonian vector field
	$$
	d_{\lambda(0)} l_i(Z_0 v_0) = d_{\lambda(0)} l_i(\vec h_0 v_0) = \sigma_{\lambda(0)}(\vec h_0 v_0,\vec l_i) = \langle \lambda(0), [f_0v_0,f_{0i}]\rangle = 0,
	$$
	where the last equality is due to involutivity of the family $f_0v_0$.
	
	Similarly one has that $Z_1u_1$ is always tangent to the image of $A(N_1)$ under the diffeomorphism $(P_{\tilde{u}}^1)^{-1}$. 
\end{proof}

\subsection{Jacobi equation and second variation}
	
\label{subsec:Jacobi}
	
	For brevity denote $\cV = \ker d_{\tilde{u}} E$. Inside $\cV$ we look at a distinguished subspace

	\begin{equation}
		\label{eq:variations}
		V = \left\{\hat v\in \cV : v_0=0,v_1=0\right\},
	\end{equation}
	which corresponds to variations that fix the end-points $q_0$, $q_1$ of an extremal curve $\gamma$ at first order. More precisely, they constitute the tangent space to the manifold of controls fixing the end-points of $\gamma$. Hence $Q|_{V}$ is the second variation of the optimal control problem with fixed end-points and there exist efficient ways of computing the index of this quadratic forms using generalisations of classical Jacobi fields~\cite[Section 21]{bookcontrol}. Our goal is to compute the difference
	$$
	\ind Q - \ind Q|_{V}
	$$
	in terms of geometric objects on the manifold $M$, which will result in formula~\eqref{eq:main_index_formula} when $N = N_0\times N_1$. The main tool for computing the difference of indices is the following folklore lemma.

	\begin{lemma}
		\label{lemma index on subspaces} Suppose that $Q$ is a continuous quadratic form on a Hilbert space. Then for any subspace $V$ of finite codimension it holds:
		\begin{equation}
			\label{eq:formula_index}
			\ind Q = \ind Q|_V + \ind Q|_{V^{\perp_Q}}  + \dim\big(V \cap V^{\perp_Q}/(V \cap \ker Q)\big).
		\end{equation}
	\end{lemma} 
	
	For us the Hilbert space will be the subspace $\mathcal{V}$. Thus
	$$
	V^{\perp_Q }  = \{\hat v \in \mathcal{V} : Q(\hat v, \hat w) = 0, \forall \hat w \in V\},
	$$ 
	and the kernel of $Q$ on $\cV$ is
	$$
	\ker Q = \{\hat v \in \cV : Q(\hat v,\hat w) = 0, \forall \hat w \in \mathcal{V}\}.
	$$

	In the following discussion we derive boundary value problems whose differential equation is a \emph{Jacobi equation}. The solution to those boundary value problems encode all the information about each term in formula~\eqref{eq:formula_index}. 
	
	By definition the subspace $V$ is the set of variations $\hat{v}=(v_0,v,v_1)$ such that $\int_0^1 Z_t v(t) dt \in \Pi$, $v_0=0,\, v_1 =0$. Moreover, since $\Pi$ is a Lagrangian subspace		 
	$$
	\int_0^1 Z_tv(t) dt \in \Pi \iff \sigma \left(\int_0^1 Z_tv(t) dt , \nu\right) =0, \,\forall \, \nu \in \Pi
	$$
	and so:
	\begin{equation*}
		\begin{split}
			\hat v \in V^{\perp_Q} &\iff Q(\hat v,\hat w) = 0, \, \forall \, \hat w \in V \\&\iff Q(\hat v,\hat w) = \int_{-1}^2 \sigma (\nu,\hat Z_t\hat w(t)) dt, \quad \forall\nu \in \Pi,  \forall \,\, \hat w \in V.
		\end{split}
	\end{equation*}
		
	Using formula~\eqref{quadratic form} we have that for almost every $t\in[0,1]$ and a vector $\nu \in \Pi$:
	\begin{equation*}
		H_t(v(t),\cdot)+\sigma \left(\int_0^t Z_\tau v(\tau) d \tau+Z_0v_0, Z_t \cdot \right)= \sigma (Z_t\cdot,\nu).
	\end{equation*}	
	By the strong Legendre condition $H_t$ is invertible. This allows us solve the equation for the variation $v$ and obtain
	$$
	v(t) = H_t^{-1}\sigma \left(Z_t\cdot,\int_0^t Z_\tau v(\tau) d \tau +Z_0v_0 + \nu\right).
	$$
	
	Set 
	$$	
	\eta (t) = \int_0^t Z_\tau v(\tau) d \tau +Z_0 v_0 + \nu.
	$$
	Differentiating $\eta$ and plugging in the expression for the variation $v$ shows that $\eta$ satisfies the following equation for almost all $t\in[0,1]$:
	\begin{equation}
		\label{jacobi equation}
		\dot{\eta}(t) = Z_tH_t^{-1}\sigma (Z_t\cdot,\eta(t)).
	\end{equation}
	This equation is known as the \textit{Jacobi equation}~\cite[Theorem 21.1]{bookcontrol}. Using the definition of $Z_0$ and Lemma~\ref{lemma:boundary_manifolds} we find that $\eta(t)$ satisfies~\eqref{jacobi equation} with $\pi_*\eta(0) \in T_{q_0}N_0$. To obtain conditions at $t=1$, we have to use the fact that $\hat{v}\in \cV$. In this case from~\eqref{eq:differential endpoint mapping} it follows that there exists $\xi \in \Pi$ such that

	\begin{equation*}
		\eta(1) = \int_0^1 Z_\tau v(\tau) d \tau +Z_0 v_0 + \nu = \xi+ \nu -Z_1v_1.
	\end{equation*} 
	Thus the variation $\hat v \in V^{\perp_Q}$ defines a function $\eta:[0,1]\to T_{\lambda(0)}(T^*M)$ which solves the following boundary value problem
	\begin{equation}
		\label{boundary value problem orthogonal complement}
		\begin{cases}
			\dot{\eta}(t) = Z_tH_t^{-1}\sigma (Z_t\cdot,\eta(t)), \\
			\pi_*\eta(0) \in T_{q_0}N_0, \quad \pi_*\eta(1) \in (\pi \circ P^{-1}_{\tilde{u}})_*(T_{q_1} A (N_1)).
		\end{cases}
	\end{equation}
	
	The second space in the boundary condition is the pull back of $N_1$ to $q_0$ using the flow of the control system with the extremal control $\tilde{u}$. From this it is immediate to compute the dimension of $V \cap V^{\perp_Q}$. It is enough to substitute $v_i=0$ in the above equations and thus consider solution starting from $\Pi$ and arriving to $\Pi$. Since the Jacobi equation derived above is exactly the same as the Jacobi equation for problem with fixed points, we immediately see that $\dim (V \cap V^{\perp_Q})$ is the multiplicity of the point $q_1$ as conjugate point.

	In the same spirit we can compute the dimension of $\ker Q \cap V$ using the Jacobi equation. We have
	\begin{equation*}
		\ker Q \cap V= \{\hat v \in V: Q(\hat v,\hat w)=0, \,\, \forall \hat w \in \mathcal{V}\}.
	\end{equation*}
	Using the same argument as above we find that for every $\nu \in \Pi$
	\begin{equation*}
		\begin{cases}
			0 = \sigma (Z_0\cdot ,\nu)\\
			Q(v,\cdot) = \sigma (Z_t\cdot ,\nu)\\
			\sigma \left(\int_0^1 Z_tv(t) dt,Z_1\cdot\right) = \sigma (Z_1\cdot,\nu) 
		\end{cases}
	\end{equation*}
	The second equation allows us to recover a solution of the Jacobi equation $\eta$ using the same argument as above, when we considered variations $\hat{v} \in V \cap V^{\perp_Q}$. The first equality gives us a condition on $\nu$ and consequently on $\eta(0)$, while the third condition give us a condition for $\eta(1)$. Namely
	$$
	\eta(0) \in \Pi \cap T_{\lambda(0)} A(N_0), \qquad \eta(1) \in \Pi \cap T_{\lambda(0)} (P^{-1}_{\tilde{u}} A(N_0)).
	$$

	The following proposition collects all the facts proved above and clarifies the correspondence between controls and solutions of the boundary value problems.
	
	\begin{prop}
		\label{prop jacobi equation subspaces}
		Consider system \eqref{jacobi equation}, to any solution $\eta$ satisfying the boundary value problem \eqref{boundary value problem orthogonal complement} we can associate a variation $v\in V^\perp_Q$ such that $\dot{\eta}(t) = Z_t v(t)$ and vice-versa, modulo solutions satisfying $\eta(0),\eta (1) \in \Pi$ and $\dot{\eta} = 0$. Moreover:
		\begin{enumerate}[(i)]
			\item elements inside  $V\cap V^{\perp_Q}$ correspond to solutions of~\eqref{jacobi equation} satisfying the boundary conditions:
			\begin{equation*}
				\eta(0) \in \Pi, \quad \eta(1) \in \Pi; 
			\end{equation*}
			\item  elements of $\ker Q \cap V$ correspond to solutions satisfying the boundary conditions:
			\begin{equation*}
				\eta(0) \in \Pi \cap T_{\lambda(0)}A (N_0), \quad \eta(1) \in \Pi \cap T_{\lambda(0)} P_{\tilde{u}}^{-1}(A (N_1));
			\end{equation*}
		\item elements in $\ker Q$ correspond to solutions of~\eqref{jacobi equation} satisfying the boundary conditions:
			\begin{equation*}
				\eta(0) \in T_{\lambda(0)}A (N_0), \quad \eta(1) \in T_{\lambda(0)} P_{\tilde{u}}^{-1}(A (N_1)).
		\end{equation*}
		\end{enumerate}
		
	\end{prop}
	
	\begin{proof}
		Following the derivation of the Jacobi equation and associated boundary conditions it only remains to prove the first part by computing the kernel of the map $\eta \mapsto v$. If $\hat v = 0$, then $v \equiv 0$ for all $t\in[0,1]$ and $v_0=0$, $v_1 = 0$. This implies that $\eta \equiv \nu$. 
		
		Vice versa, let $\eta$ be constant with $\eta(0) \in \Pi$. Since $Z_i$ are injective and have non-trivial projections to $T_{\pi(\lambda(0))}M$, it follows that $v_i=0$. Moreover $\dot{\eta}(t) = Z_tv(t) = 0$ and consequently, by definition of $V^{\perp Q}$,
		$$
		0 = Q(\hat{v},\hat{w}) = \int_0^1 H_t(v(t),w(t)) dt, \qquad \forall \hat{w} \in V.
		$$
		In particular $H_t(v(t),v(t)) =0$ for almost every $t\in[0,1]$. But then by the strong Legendre condition $v \equiv 0$. Point $(iii)$ can be obtained in a similar fashion as point $(i)$ and $(ii)$.
	\end{proof}

	As before let $\Theta$ be the differential of the Hamiltonian flow given in \cref{eq:hamilton} and $\Gamma(\Theta)$ the graph of $\Theta$.
	
	\begin{rmrk}
		\label{rmrk:flow}
		It can be shown that \eqref{jacobi equation} is closely related to the linearisation of the extremal flow along the fixed extremal $\lambda$ we are considering, see for example \cite{agrachev_stefani_bang}. It is the linearisation at $\lambda(0)$ of the Hamiltonian flow of $b^t_{\tilde{u}}(\lambda) = (H-h_{\tilde{u}(t)})\circ \Phi_t(\lambda)$ which coincides with the linearisation of $(\Phi_t)^{-1} \circ e^{tH}$. Let us denote by $\Theta_J$ the flow of the Jacobi equation~\eqref{jacobi equation} at time one and let
		$$
		\Gamma(\Theta_J) = \{(\eta(0),\eta(1)) : \eta(0) \in T_{\lambda(0)} (T^*M)\} \subset T_{\lambda(0)} (T^*M)\times T_{\lambda(0)} (T^*M) 
		$$
		be its graph. Then in this notation
		$$
		\Gamma(\Theta) = (I \times \Phi_t)_* \Gamma(\Theta_J).
		$$
	\end{rmrk}
	
	We can now compute the restriction of $Q$ to $V^{\perp_Q}$ and prove the following result. 
	
	\begin{prop}
		\label{Prop Q on V perp for separated boundary condition }
		Let $Q$ be the quadratic form of second variation for the problem~\eqref{eq:opt_ext1}-\eqref{eq:opt_ext3} and $V$ be the subspace of variations~\eqref{eq:variations}. Then
		$$ 
		\ind^-Q = \ind^{-} Q|_V + i\big(\Pi^2_{\underline \lambda},\Gamma(\Theta),T_{\underline \lambda}A(N) \big) + \dim(\Gamma(\Theta) \cap \Pi^2_{\underline \lambda}) - \dim(\Gamma(\Theta)\cap \Pi^2_{\underline \lambda} \cap T_{\underline \lambda} A(N))
		$$
		Moreover, the Maslov index of the triple can be replaced by $i\big((\Pi^2_{\underline \lambda})^{W},\Gamma(\Theta)^{W}, T_{\underline \lambda} A(N)^W\big)$ where $W = T_{\underline{\lambda}} A(N) \cap \Pi_{\underline{\lambda}}^2 $ and
		the superscript means everything is computed on the reduced subspace with respect to $W$.
	\end{prop}
	
	\begin{proof}
		
		In view of Proposition~\ref{prop jacobi equation subspaces} and Remark~\ref{rmrk:flow} it only remains to prove that 
		$$
		\ind^- Q_{V^{\perp_Q}} = i \big(\Pi^2_{\underline{\lambda}},\Gamma(\Theta),T_{\underline{\lambda}} A(N)\big).
		$$
		
		Since $(v_0,v,v_1)\in V^{\perp_Q}$, we have that 
		\begin{equation*}
			\int_{0}^{1}\left[H_t(v,w) + \sigma \left(Z_0v_0+\int_{0}^t Z_\tau v(\tau) d \tau ,Z_tw(t) \right)\right]dt = \sigma \left(\int_0^1 Z_t w(t)dt,\nu\right), \quad \forall w \in L^2[0,1], \forall \nu \in \Pi.
		\end{equation*}
		Combining the last expression with \eqref{quadratic form} gives us:
		\begin{equation}
			\label{Q on V perp}
			\begin{split}
				Q(\hat v) &= -\sigma \left(Z_0v_0+\int_0^1 Z_t v(t) dt,Z_1v_1\right) -\sigma \left(\int_0^1 Z_t v(t) dt,\nu \right)\\
				&= -\sigma (\xi,Z_1v_1) + \sigma (Z_1v_1+Z_0v_0,\nu) \\&=  -\sigma (\nu,Z_0v_0)+\sigma (\xi+\nu,-Z_1v_1)  \\
			\end{split}
		\end{equation}
		where we have used that
		$$
		\xi = Z_0v_0+Z_1v_1+\int_0^1 Z_t v(t) dt\in \Pi.
		$$
		
		From the derivation of the Jacobi equation it follows that 
		$$
		\eta(0) = \nu + Z_0v_0  \quad \eta(1) = \nu + Z_0v_0 + \int_0^1 Z_t v(t) dt = \nu+ \xi-Z_1v_1.	 
		$$	 
		Hence the restriction of $Q$ to $V^{\perp Q}$ coincides with the quadratic form 
		$$
		{m} \big(\Pi^2, \Gamma(\Theta_J), T_{\lambda(0)}A(N_0) \times T_{\lambda(0)}( \Phi_t^{-1} A(N_1) )\big).
		$$
		
		Note that $Q|_{V^{\perp_Q}}$ is actually defined on a slightly smaller space, because $Z_0v_0$ do not span the whole $T_{\lambda(0)} A(N_0)$ and similarly $Z_1v_1$ does not span $T_{\lambda(0)} (\Phi^{-1}_t A(N_1))$ correspondingly. 
		
		Nevertheless we obtain the correct Maslov form.  In fact, the map $Z_0 :\mathbb{R}^{k} \to T_{\lambda(0)} A(N_0) $ is injective and its image is transversal to $\Pi \cap T_{\lambda(0)}A(N_0)$ (and the same is true for the $Z_1$). The sum of spaces $\Pi \cap T_{\lambda(0)} A(N_0) $ and $\Pi \cap T_{\lambda(0)}(\Phi_t^{-1} A(N_1))$ lies in the kernel of the Maslov form. Removing it does not change the domain  since $\im Z_0 + \Pi = T_{\lambda(0)}A(N_0)+\Pi$ (and similarly for $Z_1$). 
		
		Hence we can either reduce by $\Pi \cap T_{\lambda(0)}A(N_0)\oplus \Pi \cap T_{\lambda(0)}(\Phi_t^{-1} A(N_1))$ or work on the original space. The index is the same.
		
		We now apply the map $I \times (\Phi_t)_*$ to each Lagrangian space inside the Maslov index of the triple above. By Remark~\ref{rmrk:flow} we get
		$$
		i \big(\Pi^2, \Gamma(\Theta_J), T_{\lambda(0)}A(N_0) \times T_{\lambda(0)}( \Phi_t^{-1} A(N_1) ) \big) = i \big(\Pi^2_{\underline \lambda},\Gamma(\Theta),T_{\underline \lambda}A(N)\big).
		$$
		
	\end{proof}

\subsection{Proof of \Cref{thm: comparison theorem}}

\label{subsec:proof}

Before proving the general formula, we prove a corollary of Proposition~\ref{Prop Q on V perp for separated boundary condition }. Assume that we have an optimal control problem \eqref{eq:control}-\eqref{eq:boundary} and two sets of possible boundary conditions:
$$
(q(0),q(1))\in N_0 \times N_1 =: N 
$$
and
$$
(q(0),q(1))\in \tilde N_0 \times \tilde N_1 =: \tilde N.
$$
and assume that a curve $\lambda: [0,1]\to TM$ is an extremal in both problems simultaneously, which simply means that $\lambda$ is a solution of the Hamiltonian system of PMP and satisfies the transversality conditions for both boundary conditions at the same time, i.e. $\lambda_i$ annihilates the sum $T_{\lambda_i}N_i +T_{\lambda_i}\tilde N_i$. A relevant example to keep in mind is when $N \subset \tilde{N}$. In this case if $\lambda$ satisfies the transversality conditions for $\tilde{N}$ it satisfies the transversality conditions for $N$ automatically.

Consider the two second variations $Q_{N}$ and $Q_{\tilde{N}}$ corresponding to the two optimal control problems with boundary conditions like above. Using Proposition~\ref{Prop Q on V perp for separated boundary condition } we can find the difference between the Morse indices of those two quadratic forms.

\begin{corollary}
	\label{cor:more_general_separated_formula}
	Using the notations of this section the following formula holds
	\begin{align}
		\ind^-Q_{\tilde N} - \ind^{-} Q_{N} &= i \big(T_{\underline \lambda}A(N),\Gamma(\Theta),T_{\underline \lambda}A(\tilde N)\big)
+ \dim(\Gamma(\Theta) \cap T_{\underline \lambda} A(N)) + \nonumber\\ &-   \dim(\Gamma(\Theta) \cap T_{\underline \lambda} A(N) \cap T_{\underline \lambda} A(\tilde N)) + \dim(T_{\pi(\underline \lambda)}N \cap T_{\pi(\underline \lambda)}\tilde{N}) - \dim  T_{\pi(\underline \lambda)} N. \label{eq:formula_several_b_conds}
	\end{align}
\end{corollary} 

\begin{proof}
	Apply Proposition~\ref{Prop Q on V perp for separated boundary condition } to get an expression for $\ind^-Q_{\tilde{N}}$ and $\ind^-Q_{N}$. Subtracting one from the other gives
	\begin{align*}
		\ind^-Q_{\tilde N} - \ind^{-} Q_{N} &= i \big(\Pi^2_{\underline \lambda},\Gamma(\Theta),T_{\underline \lambda}A(\tilde N)\big) - i \big(\Pi_{\underline \lambda},\Gamma(\Theta),T_{\underline \lambda}A(N)\big)+\\ 
		&+\dim(\Gamma(\Theta) \cap \Pi_{\underline \lambda} \cap T_{\underline \lambda} A(N)) - \dim(\Gamma(\Theta)\cap \Pi_{\underline \lambda} \cap T_{\underline \lambda} A(\tilde N)) 
	\end{align*}
	Apply formula~\eqref{eq: coboundary index} with $L_0 = \Pi^2_{\underline \lambda}$, $L_1 = \Gamma(\Theta)$, $L_2 = T_{\underline \lambda}A(\tilde N)$ and $L_3 = T_{\underline \lambda}A(N)$. After cancellations this results in
	\begin{align*}
		\ind^-Q_{\tilde N} - \ind^{-} Q_{N} &= i\big(\Gamma(\Theta),T_{\underline \lambda}A(\tilde N),T_{\underline \lambda}A(N)\big) - i \big(\Pi^2_{\underline \lambda},T_{\underline \lambda} A(\tilde N),T_{\underline \lambda} A(N)\big)+ \\
		&+\dim(\Gamma(\Theta) \cap T_{\underline \lambda} A(N)) - \dim(\Gamma(\Theta) \cap T_{\underline \lambda} A(N) \cap T_{\underline \lambda} A(\tilde N)) - \\
		& - \dim(\Pi^2_{\underline \lambda} \cap T_{\underline \lambda} A(\tilde N)) +  \dim(\Pi^2_{\underline \lambda} \cap T_{\underline \lambda} A(N) \cap T_{\underline \lambda} A(\tilde N)).
	\end{align*}
	Terms $\dim(\Gamma(\Theta) \cap T_{\underline \lambda} A(N))$, $\dim(\Gamma(\Theta) \cap T_{\underline \lambda} A(N) \cap T_{\underline \lambda} A(\tilde N))$ are already exactly as in the formula of the statement. It remains to simplify all of the remaining terms.
	
	By formula~\eqref{eq: invariance cyclic permutation index}
	$$
	i \big(\Gamma(\Theta),T_{\underline \lambda}A(\tilde N),T_{\underline \lambda}A(N)\big) = i \big(T_{\underline \lambda}A(N),\Gamma(\Theta),T_{\underline \lambda}A(\tilde N)\big)
	$$
	since we have an even permutation of subspaces inside. By \cref{lemma: maslov annullatore} we have
	$$
	i \big(\Pi^2_{\underline \lambda},T_{\underline \lambda} A(\tilde N),T_{\underline \lambda} A(N)\big) =i \big(T_{\underline \lambda} A(N),\Pi^2_{\underline \lambda},T_{\underline \lambda} A(\tilde N)\big)= 0.
	$$
	Finally, straight from the definition of an annihilator, it follows that 
	$$
	\dim(\Pi^2_{\underline \lambda} \cap T_{\underline \lambda} A(\tilde N)) = 2\dim M - \dim T_{\pi(\underline \lambda)}\tilde{N}
	$$
	and
	$$
	\dim(\Pi^2_{\underline \lambda} \cap T_{\underline \lambda} A(N) \cap T_{\underline \lambda} A(\tilde N)) = 2\dim M - \dim T_{\pi(\underline \lambda)}\tilde{N} - \dim T_{\pi(\underline \lambda)} N + \dim ( T_{\pi(\underline \lambda)}N \cap  T_{\pi(\underline \lambda)}\tilde{N}).
	$$
	Combining all of the above results in formula~\eqref{eq:formula_several_b_conds}.

\end{proof}

\begin{rmrk}
	Notice that if $N = \{q_0\}\times \{q_1\}$ we obtain exactly the formula from Proposition~\ref{prop jacobi equation subspaces} as expected. Another necessary remark is that formula~\eqref{eq:formula_several_b_conds} might seem asymmetric at first. We expect, that if we exchange $N$ and $\tilde{N}$, then the resulting right-hand side will change sign. This is not entirely obvious just from the expression itself. However, this is indeed the case, because the difference between $i\big(T_{\underline \lambda}A(N),\Gamma(\Theta),T_{\underline \lambda}A(\tilde N)\big)$ and $i \big(T_{\underline \lambda}A(\tilde N),\Gamma(\Theta),T_{\underline \lambda}A(N)\big)$ is not zero, but an expression involving dimensions of intersections of various subspaces as can be seen from formula~\eqref{eq: coboundary index}. 
	
\end{rmrk}

Now we are ready to prove \Cref{thm: comparison theorem}. We will reduce the case of general boundary conditions $(q_0,q_1) \in N \subseteq M \times M$ to the case with separated boundary conditions by introducing extra dummy variables. 
\begin{proof}[Proof of \Cref{thm: comparison theorem}]
	Consider optimal control problem \eqref{eq:control}-\eqref{eq:functional}. We can lift it to an optimal problem on $M\times M$ by considering a new control system:
	\begin{equation}
		\label{eq:control_general}
		\begin{cases}
			\dot{x} = 0,\\
			\dot{q} = f^t_{u(t)}(q),
		\end{cases}
	\end{equation}
	with boundary conditions
	\begin{equation}
		\label{eq:boundary_general}
		(x(0),q(0),x(1),q(1))\in \Delta \times N \subset M^4.
	\end{equation}
	It is clear that there is a one-to-one correspondence between admissible curves of~\eqref{eq:control}-\eqref{eq:functional} and admissible curve of~\eqref{eq:control_general}-\eqref{eq:boundary_general}. For this reason we can consider admissible curves~\eqref{eq:control_general}-\eqref{eq:boundary_general} which minimize the functional~\eqref{eq:functional}. The Hamiltonian system of PMP is then given by
	$$
	\begin{cases}
		\dot{\mu} = 0,\\
		\dot{\lambda} = \vec{H}^{t}(\lambda),
	\end{cases} \qquad \lambda,\mu \in T^*M.
	$$
	and its flow is given by $I \times {\Psi_t}$. 
	
	We can now apply directly Corollary~\ref{cor:more_general_separated_formula} to the boundary conditions $\Delta \times \tilde{N}$ and $\Delta \times N$. In order to see that everything indeed reduces to formula~\eqref{eq:main_index_formula} without writing explicitly the lengthy formula here, let us go term by term starting from the last one. Let $\underline{\lambda} = (\lambda(0),\lambda(0),\lambda(0),\lambda(1))$. We have
\begin{align*}
&\dim \left( T_{\pi(\underline\lambda)}(\Delta \times N) \cap T_{\pi(\underline\lambda)}(\Delta \times \tilde N) \right) - \dim \left( T_{\pi(\underline\lambda)}(\Delta \times N) \right) = \\
= &\dim T_{\pi(\underline\lambda)}\Delta + \dim (T_{\pi(\underline\lambda)}N \cap T_{\pi(\underline\lambda)}\tilde N) - \dim T_{\pi(\underline\lambda)}\Delta - \dim T_{\pi(\underline\lambda)}N = \dim (T_{\pi(\underline\lambda)}N \cap T_{\pi(\underline\lambda)}\tilde N) - \dim T_{\pi(\underline\lambda)}N. 
\end{align*} 

In order to deal with the last term choose Darboux coordinates around $\lambda(0)$ and $\lambda(1)$ which fix the horizontal and the vertical subspaces. For this computation we identify $T_{\lambda(0)}(T^*M)\simeq T_{\lambda(1)}(T^*M) =: \Sigma$. Let $S: \Sigma \to \Sigma$ be the map, which changes the sign of the vertical part. In Darboux coordinates it is given by the matrix
$$
\begin{pmatrix}
-1 & 0\\
0 & 1 
\end{pmatrix}.
$$
In particular, we have
$$
\sigma(S\mu_1,S\mu_2) = -\sigma(\mu_1,\mu_2), \qquad \forall\mu_1,\mu_2 \in \Sigma.
$$
Let us write down explicitly each individual subspace entering the formula. Due to our conventions of signs for the symplectic form on $(T^* M)^4$ we have $-\sigma$ on the first two copies of $T^*M$ and $\sigma$ on the last two. We have to distinguish the two definitions of annihilator here, $\hat A(N)$ is the annihilator when we use $\sigma \oplus \sigma$ (as given in \cref{def: eq annullatore}) whereas $A(N)$ is given by \cref{eq:ann_def}, and is the right object to use when the symplectic form is $(-\sigma)\oplus \sigma$. Notice once again that $S\times I(\hat{A}(N)) = A(N)$.
\begin{align}
T_{ \underline\lambda}A(\Delta \times N) &= \{(S\xi,\xi, \nu_1, \nu_2): \xi \in \Sigma, (\nu_1,\nu_2)\in T_{\underline{\lambda}}\hat A(N)\},\nonumber\\
T_{ \underline\lambda}A(\Delta \times \tilde N) &= \{(S\tilde \xi,\tilde \xi, \tilde \nu_1, \tilde \nu_2): \tilde \xi \in \Sigma, (\tilde \nu_1,\tilde \nu_2)\in T_{\underline{\lambda}}\hat A(\tilde N)\},\label{eq:annihilators}\\
\Gamma(I\times \Theta) &= \{(\eta_1,\eta_2, \eta_1, \Theta \eta_2):\eta_1,\eta_2 \in \Sigma\}.\nonumber
\end{align}

From expressions in~\eqref{eq:annihilators} it directly follows that
	\begin{align*}
		\dim(\Gamma(I \times \Theta) \cap T_{ \underline\lambda} A(\Delta \times N))  &= \dim(\Gamma(\Theta) \cap T_{\underline \lambda} A(N)) \\
		\dim(\Gamma(I \times \Theta) \cap T_{\underline\lambda} A(\Delta \times N) \cap T_{ \underline\lambda} A(\Delta \times \tilde N)) &=   \dim(\Gamma(\Theta) \cap T_{\underline \lambda} A(N) \cap T_{\underline \lambda} A(\tilde N)).
	\end{align*}

In order to simplify the Maslov index term, we note that the intersection of annihilators contains the following isotropic subspace
$$
W = \{(S\xi,\xi,0,0): \xi \in \Sigma\}.
$$
For this reason we can perform a reduction to the space $W^\perp / W$. We have
$$
W^\perp = \{(S\xi_2,\xi_2,\xi_3,\xi_4): \xi_i \in \Sigma\}.
$$
Thus we can identify $W^\perp/W$ with the image of the projection $\pi_1: \Sigma^4 \to \Sigma^2$ to the last two terms $(\Sigma \times \Sigma, \sigma \oplus \sigma)$.  Let us consider the space $(W^\perp + W) \cap \Gamma(I \times \Theta)$, it is straight forward to check that its projection is the subspace $\{(\eta,\Theta \eta): \eta \in \Sigma\}$.  We can now calculate the Maslov form on the reduced space. First of all we write down the equation defining the subspace:
$$
\begin{cases}
\nu_1 + \tilde \nu_1 = \eta, \\
\nu_2 + \tilde{\nu}_2 = \Theta \eta.
\end{cases}
$$
where  $(\nu_1,\nu_2)\in T_{\underline{\lambda}}\hat A(N)$, $
(\tilde \nu_1,\tilde \nu_2)\in T_{\underline{\lambda}}\hat A(\tilde N)$, 
$\eta \in \Sigma$. It follows that
$$
m\left((\eta,\Theta\eta)\right) = \sigma(\nu_1,\tilde \nu_1)+\sigma(\nu_2,\tilde \nu_2) = -\sigma(S\nu_1,S\tilde \nu_1)+\sigma(\nu_2,\tilde \nu_2),   
$$

 But this is exactly the Maslov form $m \big(T_{\underline \lambda}A(N),\Gamma(\Theta),T_{\underline \lambda}A (\tilde N)\big)$. Hence
\begin{equation*}
		\label{eq:two_maslov_equal}
		i\big(T_{ \underline\lambda}A(\Delta \times N),\Gamma(I \times \Theta),T_{ \underline\lambda}A (\Delta \times \tilde N)\big) = i \big(T_{\underline \lambda}A(N),\Gamma(\Theta),T_{\underline \lambda}A (\tilde N)\big).
	\end{equation*}

\end{proof}

\bibliographystyle{plain}
\bibliography{ref}

\begin{thebibliography}{10}

\bibitem{abbondandolo_infinite_morse}
Alberto Abbondandolo and Pietro Majer.
\newblock A {M}orse complex for infinite dimensional manifolds. {I}.
\newblock {\em Adv. Math.}, 197(2):321--410, 2005.

\bibitem{noteNLS}
R.~Adami.
\newblock Ground states for {NLS} on graphs: a subtle interplay of metric and
  topology.
\newblock {\em Math. Model. Nat. Phenom.}, 11(2):20--35, 2016.

\bibitem{NLS2}
Riccardo Adami, Claudio Cacciapuoti, Domenico Finco, and Diego Noja.
\newblock Variational properties and orbital stability of standing waves for
  {NLS} equation on a star graph.
\newblock {\em J. Differential Equations}, 257(10):3738--3777, 2014.

\bibitem{NLS1}
Riccardo Adami, Enrico Serra, and Paolo Tilli.
\newblock N{LS} ground states on graphs.
\newblock {\em Calc. Var. Partial Differential Equations}, 54(1):743--761,
  2015.

\bibitem{NLs3}
Riccardo Adami, Enrico Serra, and Paolo Tilli.
\newblock Threshold phenomena and existence results for {NLS} ground states on
  metric graphs.
\newblock {\em J. Funct. Anal.}, 271(1):201--223, 2016.

\bibitem{agrachev_quadratic_paper}
A.~A. Agrach\"{e}v.
\newblock Quadratic mappings in geometric control theory.
\newblock In {\em Problems in geometry, {V}ol. 20 ({R}ussian)}, Itogi Nauki i
  Tekhniki, pages 111--205. Akad. Nauk SSSR, Vsesoyuz. Inst. Nauchn. i Tekhn.
  Inform., Moscow, 1988.
\newblock Translated in J. Soviet Math. {{\bf{5}}1} (1990), no. 6, 2667--2734.

\bibitem{beschastnyi_1d}
Andrei Agrachev and Ivan Beschastnyi.
\newblock Jacobi fields in optimal control: one-dimensional variations.
\newblock {\em J. Dyn. Control Syst.}, 26(4):685--732, 2020.

\bibitem{beschastnyi_morse}
Andrei Agrachev and Ivan Beschastnyi.
\newblock Jacobi fields in optimal control: {M}orse and {M}aslov indices.
\newblock {\em Nonlinear Anal.}, 214:Paper No. 112608, 47, 2022.

\bibitem{bookcontrol}
Andrei~A. Agrachev and Yuri~L. Sachkov.
\newblock {\em Control theory from the geometric viewpoint}, volume~87 of {\em
  Encyclopaedia of Mathematical Sciences}.
\newblock Springer-Verlag, Berlin, 2004.
\newblock Control Theory and Optimization, II.

\bibitem{agrachev_stefani_bang}
Andrei~A. Agrachev, Gianna Stefani, and Pierluigi Zezza.
\newblock Strong optimality for a bang-bang trajectory.
\newblock {\em SIAM J. Control Optim.}, 41(4):991--1014, 2002.

\bibitem{andrey_sachkov_elastica}
A.~A. Ardentov and Yu.~L. Sachkov.
\newblock Solution of {E}uler's elastica problem.
\newblock {\em Avtomat. i Telemekh.}, (4):78--88, 2009.

\bibitem{aronna_bang_singular}
M.~Soledad Aronna, J.~Fr\'{e}d\'{e}ric Bonnans, Andrei~V. Dmitruk, and Pablo~A.
  Lotito.
\newblock Quadratic order conditions for bang-singular extremals.
\newblock {\em Numer. Algebra Control Optim.}, 2(3):511--546, 2012.

\bibitem{baryshnikov}
Yu.~M. Baryshnikov.
\newblock Indices for extremal embeddings of {$1$}-complexes.
\newblock In {\em Theory of singularities and its applications}, volume~1 of
  {\em Adv. Soviet Math.}, pages 137--144. Amer. Math. Soc., Providence, RI,
  1990.

\bibitem{berkolaiko}
Gregory Berkolaiko and Peter Kuchment.
\newblock {\em Introduction to quantum graphs}, volume 186 of {\em Mathematical
  Surveys and Monographs}.
\newblock American Mathematical Society, Providence, RI, 2013.

\bibitem{overviewNLS}
William Borrelli, Raffaele Carlone, and Lorenzo Tentarelli.
\newblock An overview on the standing waves of nonlinear schr\"odinger and
  dirac equations on metric graphs with localized nonlinearity.
\newblock {\em Symmetry}, 11(2), 2019.

\bibitem{Bott}
R.~Bott.
\newblock On the iteration of closed geodesics and the {S}turm intersection
  theory.
\newblock {\em Comm. Pure Appl. Math.}, 9:171--206, 1956.

\bibitem{NLS4}
Claudio Cacciapuoti, Domenico Finco, and Diego Noja.
\newblock Ground state and orbital stability for the {NLS} equation on a
  general starlike graph with potentials.
\newblock {\em Nonlinearity}, 30(8):3271--3303, 2017.

\bibitem{cox_jones_schroedinger}
Graham Cox, Christopher K. R.~T. Jones, Yuri Latushkin, and Alim Sukhtayev.
\newblock The {M}orse and {M}aslov indices for multidimensional
  {S}chr\"{o}dinger operators with matrix-valued potentials.
\newblock {\em Trans. Amer. Math. Soc.}, 368(11):8145--8207, 2016.

\bibitem{cox_jones_morse_bounded_domains}
Graham Cox, Christopher K. R.~T. Jones, and Jeremy~L. Marzuola.
\newblock A {M}orse index theorem for elliptic operators on bounded domains.
\newblock {\em Comm. Partial Differential Equations}, 40(8):1467--1497, 2015.

\bibitem{CushmannDuistermaatIteration}
R.~Cushman and J.~J. Duistermaat.
\newblock The behavior of the index of a periodic linear {H}amiltonian system
  under iteration.
\newblock {\em Advances in Math.}, 23(1):1--21, 1977.

\bibitem{bookgosson}
Maurice de~Gosson.
\newblock {\em Symplectic geometry and quantum mechanics}, volume 166 of {\em
  Operator Theory: Advances and Applications}.
\newblock Birkh\"{a}user Verlag, Basel, 2006.
\newblock Advances in Partial Differential Equations (Basel).

\bibitem{NLSperiodic}
Simone Dovetta.
\newblock Mass-constrained ground states of the stationary {NLSE} on periodic
  metric graphs.
\newblock {\em NoDEA Nonlinear Differential Equations Appl.}, 26(5):Paper No.
  30, 30, 2019.

\bibitem{DuistermaatIntersectionIndex}
J.~J. Duistermaat.
\newblock On the {M}orse index in variational calculus.
\newblock {\em Advances in Math.}, 21(2):173--195, 1976.

\bibitem{dunne}
Gerald Dunne.
\newblock Functional determinants in quantum field theory.
\newblock {\em Journal of Physics A: Mathematical and Theoretical}, 41, 12
  2007.

\bibitem{piccone_conjugate}
Miguel~\'{A}ngel Javaloyes and Paolo Piccione.
\newblock Comparison results for conjugate and focal points in
  semi-{R}iemannian geometry via {M}aslov index.
\newblock {\em Pacific J. Math.}, 243(1):43--56, 2009.

\bibitem{latushkin_index_on_graphs}
Yuri Latushkin and Selim Sukhtaiev.
\newblock An index theorem for {S}chr\"{o}dinger operators on metric graphs.
\newblock In {\em Analytic trends in mathematical physics}, volume 741 of {\em
  Contemp. Math.}, pages 105--119. Amer. Math. Soc., [Providence], RI, [2020]
  \copyright 2020.

\bibitem{ludewig}
Matthias Ludewig.
\newblock Heat kernel asymptotics, path integrals and infinite-dimensional
  determinants.
\newblock {\em J. Geom. Phys.}, 131:66--88, 2018.

\bibitem{portaluri_morse_geodesics}
Monica Musso, Jacobo Pejsachowicz, and Alessandro Portaluri.
\newblock A {M}orse index theorem for perturbed geodesics on semi-{R}iemannian
  manifolds.
\newblock {\em Topol. Methods Nonlinear Anal.}, 25(1):69--99, 2005.

\bibitem{osmolovskii}
Nikolai~P. Osmolovskii.
\newblock On second-order necessary conditions for broken extremals.
\newblock {\em J. Optim. Theory Appl.}, 164(2):379--406, 2015.

\bibitem{portaluri_waterstraat}
Alessandro Portaluri and Nils Waterstraat.
\newblock A {M}orse-{S}male index theorem for indefinite elliptic systems and
  bifurcation.
\newblock {\em J. Differetial Equations}, 258(5):1715--1748, 2015.

\bibitem{schattler-ledzewicz-book}
Heinz Sch\"{a}ttler and Urszula Ledzewicz.
\newblock {\em Geometric optimal control}, volume~38 of {\em Interdisciplinary
  Applied Mathematics}.
\newblock Springer, New York, 2012.
\newblock Theory, methods and examples.

\bibitem{swanson}
R.~C. Swanson.
\newblock Fredholm intersection theory and elliptic boundary deformation
  problems. {I}.
\newblock {\em J. Differential Equations}, 28(2):189--201, 1978.

\bibitem{waterstraat_kmorse}
Nils Waterstraat.
\newblock A {$K$}-theoretic proof of the {M}orse index theorem in
  semi-{R}iemannian geometry.
\newblock {\em Proc. Amer. Math. Soc.}, 140(1):337--349, 2012.

\end{thebibliography}

\end{document}